\documentclass[11pt,reqno]{amsart}

\makeatletter
\newcommand*{\rom}[1]{\expandafter\@slowromancap\romannumeral #1@}
\makeatother

\usepackage{amsmath,amsfonts,amssymb}
\usepackage{mathtools}
\usepackage[center]{caption}
\usepackage{amsmath,amssymb,amsfonts,mathrsfs,verbatim,enumitem,amsthm, graphicx}
\usepackage{color}
\usepackage[all]{xy}
\usepackage{dirtytalk}
\usepackage{centernot, xr, accents}
\usepackage{graphicx}
\usepackage[export]{adjustbox}
\usepackage{hyperref, cleveref}
\usepackage{romannum}

\numberwithin{equation}{section}

\usepackage{hyperref}
\hypersetup{
	colorlinks,
	citecolor=red,
	filecolor=black,
	linkcolor=blue,
	urlcolor=black
}

\theoremstyle{plain}
\newtheorem{thm}{Theorem}[section]

\theoremstyle{definition}
\newtheorem{rem}[thm]{Remark} 
\newtheorem{convention}{Convention}

\def \hf{\hspace*{0.5cm}}

\def \hf{\hspace*{0.5cm}}

\makeatletter
\renewcommand{\tocsection}[3]{%
  \indentlabel{\@ifnotempty{#2}{\bfseries\ignorespaces#1 #2\quad}}\bfseries#3}
\renewcommand{\tocsubsection}[3]{%
  \indentlabel{\@ifnotempty{#2}{\ignorespaces#1 #2\quad}}#3}
\newcommand\@dotsep{4.5}
\def\@tocline#1#2#3#4#5#6#7{\relax
  \ifnum #1>\c@tocdepth % then omit
  \else
    \par \addpenalty\@secpenalty\addvspace{#2}%
    \begingroup \hyphenpenalty\@M
    \@ifempty{#4}{%
      \@tempdima\csname r@tocindent\number#1\endcsname\relax
    }{%
      \@tempdima#4\relax
    }%
    \parindent\z@ \leftskip#3\relax \advance\leftskip\@tempdima\relax
    \rightskip\@pnumwidth plus1em \parfillskip-\@pnumwidth
    #5\leavevmode\hskip-\@tempdima{#6}\nobreak
    \leaders\hbox{$\m@th\mkern \@dotsep mu\hbox{.}\mkern \@dotsep mu$}\hfill
    \nobreak
    \hbox to\@pnumwidth{\@tocpagenum{\ifnum#1=1\bfseries\fi#7}}\par% <-- \bfseries for \section page
    \nobreak
    \endgroup
  \fi}
\AtBeginDocument{%
\expandafter\renewcommand\csname r@tocindent0\endcsname{0pt}
}
\def\l@subsection{\@tocline{2}{0pt}{2.5pc}{5pc}{}}
\makeatother

\AtBeginDocument{\pagenumbering{arabic}}

\begin{document}

\title[On a fiber bundle version of the Caporaso-Harris formula]{On a fiber bundle version of the Caporaso-Harris formula}

\author[I. Biswas]{Indranil Biswas}

\address{Department of Mathematics, Shiv Nadar University, NH91, Tehsil
Dadri, Greater Noida, Uttar Pradesh 201314, India}

\email{indranil.biswas@snu.edu.in, indranil29@gmail.com}

\author[A. Choudhury]{Apratim Choudhury}

\address{Institut f{\"u}r Mathematik, Humboldt-Universit{\"a}t zu Berlin,
Unter den Linden 6, Berlin- 10099, Germany}
\email{apratim.choudhury@hu-berlin.de}

\author[N. Das]{Nilkantha Das}
\address{Stat-Math Unit, Indian Statistical Institute, B.T.Road, Kolkata-700108, India.}
\email{dasnilkantha17@gmail.com}

\author[Ritwik Mukherjee]{Ritwik Mukherjee}
\address{School of Mathematical Sciences, National Institute of Science Education and Research, Bhubaneswar, An OCC of Homi Bhabha National Institute,
Khurda 752050, Odisha, India.}
\email{ritwikm@niser.ac.in}


\date{}

\begin{abstract} 
The Caporaso-Harris formula gives a recursive algorithm to enumerate $\delta$-nodal 
degree $d$ curves in $\mathbb{P}^2$. The recursion is obtained in terms of curves of lower degree that 
are tangent to a given divisor. This paper presents two generalizations of this method.    
The first result is on enumeration of one cuspidal curves on $\mathbb{P}^2$, and the second result 
is an extension to the fiber bundle setting. We solve the question of counting the characteristic
number planar nodal cubics in $\mathbb{P}^3$ by extending the idea in \cite{CH}.
\end{abstract}

\maketitle

\tableofcontents

\section{Introduction}

The Caporaso-Harris formula was a landmark achievement in the field of enumerative geometry. 
Very simply stated, it solves the following problem: how many $\delta$-nodal degree $d$ curves are there
on $\mathbb{P}^2$ that pass through $\frac{d(d+3)}{2}-\delta$ generic points? 
Their formula \cite[p.~348, Theorem 1.1]{CH} is a recursion in terms of curves of lower degree 
which can be tangent (of higher order) to a given line. More generally, they obtain a formula 
for the characteristic number of curves in $\mathbb{P}^2$, having $\delta$-nodes, 
that intersect a given line tangentially at multiple points. 

The underlying geometric idea behind the Caporaso-Harris formula is a degeneration argument. The idea is to move the constraints 
(points) to the line and compare the resulting numbers. Moving the point to the line usually results in a degenerate locus 
involving a curve of a lower degree. Iterating this process, one obtains a recursive formula. 

A natural question is to count curves with higher singularities. Addressing this, there are all completely 
different approaches and have nothing to do with the degeneration argument of \cite{CH}. 
It is therefore natural to ask: \textbf{can the degeneration argument of \cite{CH} be used to enumerate curves 
with higher singularities?}               

The goal of this article is twofold. This paper answers the above question when the singularity is a cusp. We expect that with further 
effort, the degeneration argument can be carried forward to enumerate curves with even more degenerate singularities 
(such as tacnodes, triple-points etc).  

The second result of this paper is explained with some background to put it in perspective. 
A typical question in enumerative geometry is about the number of objects (curves, for example) in a fixed space satisfying  certain constraints.
In \cite{KP1}, \cite{Kl.P} and \cite{KP_Sigma},
Kleiman and Piene introduce a very natural generalization to this question.  
Suppose we have a \textbf{family} of varieties (i.e., a fiber bundle). 
How many curves are there in this fiber bundle such that each curve lies inside a single fiber and 
satisfies certain constraints? A prototypical example to keep in mind is as follows: enumeration of planar degree $d$-curves 
in $\mathbb{P}^3$ (i.e., curves, whose image lies inside some $\mathbb{P}^2$), that satisfy certain constraints. 
In \cite{TL}, Ties Laarakker obtains a formula for the number of $\delta$-nodal planar degree $d$ 
curves in $\mathbb{P}^3$, that intersect the correct number of generic lines.
He solves that question by extending the idea of Kool, Shende and Thomas,
\cite{KST}, where they solve the question of enumerating $\delta$-nodal curves in a fixed $\mathbb{P}^2$. 

With this in the background, it is natural to ask whether the degeneration method of Caporaso-Harris can be used to recover the 
numbers obtained by Ties Laarakker.

The main problem in extending this idea to count nodal planar cubics in $\mathbb{P}^3$ is as follows: what exactly would be 
the 
analogue of the process of ``moving points on a line'' in the case of a moving family of $\mathbb{P}^2$? 
The solution lies in interpreting the Caporaso-Harris formula on the level of cycles in a suitable space 
(something more general than \cite[p.~351, Theorem 1.2]{CH}). Once we interpret the formula suitably on the 
level of cycles, extending it to a family version is conceptually straightforward. 

We show that by extending the degeneration argument of Caporaso-Harris, 
it is possible to compute the characteristic number of nodal planar cubics in $\mathbb{P}^3$. 
This is the second result of this paper.
The numbers obtained agree 
with those in \cite{TL} and also agree with the numbers obtained by the third and fourth author by
alternative means 
\cite{BMS_R_N}, \cite{BMS}. We hope that the simple example of nodal planar cubics 
worked out in this paper will enable one in future to solve the question of obtaining all the numbers 
computed by Ties Laarakker in \cite{TL}.

\section{Main Results}\label{mr_section}

Let us now state the two main results of this paper.  
Let $\mathsf{N}_d(\mathsf{A}_1)$ denote the number of nodal degree $d$ curves in $\mathbb{P}^2$ 
passing through 
$\delta_d-1$ generic points, where
\begin{equation}\label{deld}
\delta_d\,\,:=\,\, \frac{d(d+3)}{2}.
\end{equation} 
Denote by $\mathsf{N}_d(\mathsf{T}_1)$ the  
number of degree $d$ curves in $\mathbb{P}^2$, tangent to a given line 
(of first order) and passing through $\delta_d-1$ generic points. 
A special case of \cite[p.~348, Theorem 1.1]{CH} is 
\begin{align}
\mathsf{N}_d(\mathsf{A}_1) & \,=\,\, (2d-1) + \mathsf{N}_{d-1}(\mathsf{A}_1) + 2 \mathsf{N}_{d-1}(\mathsf{T}_1). \label{na1_CH}
\end{align} 
Denote by
$\mathsf{N}_d(\mathsf{A}_1^{\mathsf{L}})$ to be the number of degree $d$ curves 
with a node lying on a generic line 
and 
passing through 
$\delta_d-2$ generic points. Also, let $\mathsf{N}_d(\mathsf{T}_1^{\mathsf{Pt}})$ denote the  
number of degree $d$ curves in $\mathbb{P}^2$ tangent to a given line
at a given point and passing through $\delta_d-2$ generic points. Denote the number of degree $d$ curves in $\mathbb{P}^2$ 
having a cusp and passing through 
$\delta_d-2$ generic points by
$\mathsf{N}_d(\mathsf{A}_2)$. Finally, let $\mathsf{N}_d(\mathsf{T}_2)$ denote the  
number of degree $d$ curves in $\mathbb{P}^2$ tangent to a given line 
(to second order) and passing through $\delta_d-2$ generic points. 

The first main result (it is proved in \Cref{CH_one_node_review}): 

\begin{thm}
\label{mr1}
With the notations as above, the following two identities hold for $d\geq 3$:
\begin{align}
\mathsf{N}_d(\mathsf{A}_1^{\mathsf{L}}) & \,\,=\,\, 1 + \mathsf{N}_{d-1}(\mathsf{A}_1^{\mathsf{L}}) + 2
\mathsf{N}_{d-1}(\mathsf{T}_1^{\mathsf{Pt}}),\label{mr_11}\\ 
\mathsf{N}_d(\mathsf{A}_2) & \,=\, 3\mathsf{N}_{d-1}(\mathsf{T}_{1}) +
3d\mathsf{N}_{d-1}(\mathsf{T}_{1}^{\mathsf{Pt}}) + \mathsf{N}_{d-1}(\mathsf{A}_2) + 
3\mathsf{N}_{d-1}(\mathsf{A}_1^{\mathsf{L}})+ 
2 \mathsf{N}_{d-1}(\mathsf{T}_2).\label{mr_12}
\end{align} 
\end{thm}

First of all, note that 
$\mathsf{N}_d(\mathsf{T}_k)$ and $\mathsf{N}_d(\mathsf{T}_k^{\mathsf{Pt}})$ 
can be computed from the Caporaso-Harris formula for all $d \,\geq\, 1$ combined with 
\cref{mr_11,mr_12}. 

Next note that  
$\mathsf{N}_2(\mathsf{A}_1^{\mathsf{L}}) \,=\,3$. 
To see why this is so, observe that 
$\mathsf{N}_2(\mathsf{A}_1^{\mathsf{L}})$ is the number of nodal conics 
passing through three generic points, with the node lying on a generic line. 
This is simply counting the number of pairs of lines passing through  
three generic points, with the point of intersection lying on a generic line. 
We can place one of the lines through two of the three points. There are three ways to do 
that. Once we place the first line through two of the three given points, 
the second line gets fixed (since it is the unique line passing through the third remaining 
point and the point of intersection of the first line with the line on which the nodal 
point has to lie). Hence, $\mathsf{N}_2(\mathsf{A}_1^{\mathsf{L}})$ is equal to $3$. 

Finally, it can shown that $\mathsf{N}_2(\mathsf{A}_2) \,=\,0$. 
To see this, we note that $\mathsf{N}_2(\mathsf{A}_2)$ is the number of conics passing through three
generic points having a cusp. Since a double line does not pass through three generic 
points, it follows that the number is zero.  

These initial conditions enable us to use \cref{mr_12} to recursively compute 
$\mathsf{N}_d(\mathsf{A}_2)$ for all $d$. 
For the convenience of the reader we display the vales of 
$\mathsf{N}_d(\mathsf{A}_2)$ for a the first few values of $d$  
\begin{center}
\begin{tabular}{|c|c|c|c|c|c|c|} 
\hline
$d$  &$3$ &$4$ & $5$ & $6$ & $7$ & $8$  \\
\hline 
$\mathsf{N}_d(\mathsf{A}_2)$ &$24$ &$72$ & $144$ & $240$ & $360$ & $504$ \\
\hline 
\end{tabular}
\captionof{table}{\textnormal{~~}} 
\label{IRCAA_tab}
\end{center}

It may be mentioned that although the computation of $\mathsf{N}_d(\mathsf{A}_2)$ has been done earlier by 
other means (see for example the approach 
in Eisenbud's book 
\cite[Section 11.4]{EBH}), 
the degeneration
formula \cref{mr_12} itself is new. The numbers displayed in \Cref{IRCAA_tab} agree with 
those obtained by the formula given in \cite[Proposition 11.4, Page 416]{EBH}. 

The second result (it is proved in \Cref{CH_pl_nodal_details}):

\begin{thm}
\label{mr2}
Let $\mathsf{N}_3^{\mathsf{Pl}}(\mathsf{A}_1)(r, s)$
denote the number of nodal planar cubics in $\mathbb{P}^3$ that intersect $r$ generic lines 
and $s$ generic points, where $r+2s\,=\,11$. 
By extending the degeneration argument of Caporaso-Harris, this number can be computed 
and 
is found to be  
\begin{center}
\begin{tabular}{|c|c|c|c|c|} 
\hline
$(r,s)$ &$(11,0)$ &$(9,1)$ &$(7,2)$ & $(5,3)$  \\
\hline 
$\mathsf{N}_3^{\mathsf{Pl}}(\mathsf{A}_1)(r, s)$ & $12960$ &$1392$ &$144$ & $12$  \\
\hline 
\end{tabular}
\captionof{table}{\textnormal{~~}} 
\label{mr2_numbers}
\end{center}
\end{thm}     

Note that $\mathsf{N}_3^{\mathsf{Pl}}(\mathsf{A}_1)(r, s)$ is automatically zero if 
$s$ is greater than or equal to four
because four generic points do not lie in a plane.    
The values given in \Cref{mr2_numbers} are in agreement 
with what is obtained in the papers \cite{BMS_R_N}, \cite{TL} and \cite{BMS}.
We have written a mathematica program to implement 
the computations of \Cref{CH_pl_nodal_details}. 
The program is available on the fourth author's homepage 
\[ \textnormal{\url{https://www.sites.google.com/site/ritwik371/home}} \]

\section{Caporaso-Harris formula for one nodal curves}\label{CH_one_node_review}

In this section the following question is addressed:

\begin{center}
{\it How many degree $d$ curves are there in $\mathbb{P}^2$ that have a node and pass through $\delta_d-1$ 
generic points? (see \cref{deld})}
\end{center}

Let $\mathsf{N}_d(\mathsf{A}_1, k)$ be the number of nodal degree $d$ 
curves, passing through $\delta_d-1-k$ generic points and 
$k$ points on a line; also, $\mathsf{N}_d(\mathsf{A}_1,0)$ is
written as $\mathsf{N}_d(\mathsf{A}_1)$ for short.

It is implied by \cite[p.~348, Theorem 1.1]{CH} that 
\begin{align}
\mathsf{N}_d(\mathsf{A}_1) & \,=\, \mathsf{N}_d(\mathsf{A}_1, 1), \nonumber \\ 
\mathsf{N}_d(\mathsf{A}_1, 1) & \,=\, \mathsf{N}_d(\mathsf{A}_1, 2), \nonumber \\ 
\ldots & \nonumber \\ 
\mathsf{N}_d(\mathsf{A}_1, d-2) & \,=\, \mathsf{N}_d(\mathsf{A}_1, d-1). \label{CH_degen_arg}
\end{align}   
In other words, 
\begin{align} 
\mathsf{N}_d(\mathsf{A}_1) & \,\,=\,\, \mathsf{N}_d(\mathsf{A}_1, d-1). \label{na1_d_minus_1}
\end{align}
Hence, the number of nodal degree $d$ curves passing through $\delta_d-1$ generic points coincides with
the number of nodal degree $d$ curves passing through $\delta_d-d$ generic points and $(d-1)$ 
collinear points. Intuitively, this is saying that if we move some of the point constraints to a line, 
the resulting count of curves is still the same.  

Unwinding the formula \cite[p.~348, Theorem 1.1]{CH} further, it follows that 
\begin{align}
\mathsf{N}_d(\mathsf{A}_1, d-1) & \,\neq\,  \mathsf{N}_d(\mathsf{A}_1, d). \label{CH_neq_deg}
\end{align} 
Let us try to understand this via the example of cubics. Equation \eqref{na1_d_minus_1} says that the number 
of nodal cubics through $8$ generic points is the same as the number of nodal cubics through $6$ generic points 
and $2$ points on a line. Equation \eqref{CH_neq_deg} on the other hand says that 
this number is \textit{not} the same as the number of nodal cubics through $5$ generic points and $3$ points on a line. 
The underlying geometric reason for \cref{CH_neq_deg} 
is as follows: when the $d^{\textnormal{th}}$ point is moved to the line, we get a degenerate contribution from a curve of 
degree $d-1$. 
How then are the two numbers $\mathsf{N}_d(\mathsf{A}_1, d-1)$ and 
$\mathsf{N}_d(\mathsf{A}_1, d)$ related? 

Denote the number of smooth degree $d-1$ curves passing through 
$\delta_{d-1}$ generic points by $\mathsf{N}_{d-1}(\mathsf{S})$. This number is obviously $1$, but let us answer the earlier question in terms of 
$\mathsf{N}_{d-1}(\mathsf{S})$ (which will make the equation geometrically more transparent). 
Unwinding \cite[p.~348, Theorem 1.1]{CH}, it follows that
\begin{align}
\mathsf{N}_d(\mathsf{A}_1, d-1) & \,=\, \mathsf{N}_d(\mathsf{A}_1, d) + (d-1)\mathsf{N}_{d-1}(\mathsf{S}). \label{na1_sm}  
\end{align} 
Using \cref{na1_d_minus_1,na1_sm}, the question of computing $\mathsf{N}_d(\mathsf{A}_1)$ 
has got reduced to a question of computing $\mathsf{N}_{d-1}(\mathsf{S})$ and $\mathsf{N}_d(\mathsf{A}_1, d)$. 
Since we know that $\mathsf{N}_{d-1}(\mathsf{S})\,=\, 1$, the question has now been reduced to the computation of 
$\mathsf{N}_d(\mathsf{A}_1, d)$. 

Recall that $\mathsf{N}_{d-1}(\mathsf{T}_1)$ is defined to be the number of smooth degree $d-1$ curves
passing through $\delta_{d-1}-1$ generic points and tangent to a line.
Unwinding \cite[Theorem 1.1]{CH} further, we get that
\begin{align}
\mathsf{N}_d(\mathsf{A}_1, d) & \,=\, \mathsf{N}_{d-1}(\mathsf{A}_1) + d \mathsf{N}_{d-1}(\mathsf{S}) + 2 \mathsf{N}_{d-1}(\mathsf{T}_1). 
\label{na1_d_pts}
\end{align}
\Cref{na1_d_minus_1,na1_sm,na1_d_pts} combined with the fact that 
$\mathsf{N}_{d-1}(\mathsf{S}) \,=\, 1$ give \cref{na1_CH}.   

Let $\mathsf{N}_d(\mathsf{A}_1^{\mathsf{L}}, k)$ denote the number of nodal degree $d$ 
curves, with the nodal point lying on a line and passing through $\delta_d-2-k$ generic points and 
$k$ points on a line. Note that the line on which the nodal point lies and the line on which the 
$k$ points lie are different. Again, when $k\,=\,0$, we omit writing the $k$, i.e., 
$\mathsf{N}_d(\mathsf{A}_1^{\mathsf{L}})\,:=\, \mathsf{N}_d(\mathsf{A}_1^{\mathsf{L}},0)$. 
Analogous to how \cref{CH_degen_arg} is proved, it will similarly be shown that 
\begin{align}
\mathsf{N}_d(\mathsf{A}_1^{\mathsf{L}}) & \,=\, \mathsf{N}_d(\mathsf{A}_1^{\mathsf{L}}, 1), \nonumber \\ 
\mathsf{N}_d(\mathsf{A}_1^{\mathsf{L}}, 1) & \,=\, \mathsf{N}_d(\mathsf{A}_1^{\mathsf{L}}, 2), \nonumber \\ 
\ldots & \nonumber \\ 
\mathsf{N}_d(\mathsf{A}_1^{\mathsf{L}}, d-2) & \,=\, \mathsf{N}_d(\mathsf{A}_1^{\mathsf{L}}, d-1). \label{CH_degen_arg_L}
\end{align}   
In other words, 
\begin{align} 
\mathsf{N}_d(\mathsf{A}_1^{\mathsf{L}}) & \,\,=\,\, \mathsf{N}_d(\mathsf{A}_1^{\mathsf{L}}, d-1). \label{na1_d_minus_1_L}
\end{align}
Next, similar to how \cref{na1_sm} is proved, it will be shown that 
\begin{align}
\mathsf{N}_d(\mathsf{A}_1^{\mathsf{L}}, d-1) & \,\,= \,\,\mathsf{N}_d(\mathsf{A}_1^{\mathsf{L}}, d) + 
\mathsf{N}_{d-1}(\mathsf{S}). \label{na1_sm_L}  
\end{align} 
Finally, similar to how \cref{na1_d_pts} is proved, it will be shown that  
\begin{align}
\mathsf{N}_d(\mathsf{A}_1^{\mathsf{L}}, d) & \,\,=\,\, 
\mathsf{N}_{d-1}(\mathsf{A}_1^{\mathsf{L}}) + 2 \mathsf{N}_{d-1}(\mathsf{T}_1^{\textsf{Pt}}). 
\label{na1_d_pts_L}
\end{align}
\Cref{na1_d_minus_1_L,na1_sm_L,na1_d_pts_L} 
combined together with the fact that $\mathsf{N}_{d-1}(\mathsf{S})\,=\,1$, produce \cref{mr_11}. 

Let us now explain this entire degeneration process in greater detail.
Denote by $\mathcal{D}_d$ the projective space of dimension $\delta_d$
given by the curves of degree $d$ on
the projective plane. Denote the hyperplane class in $\mathcal{D}_d$ by $y_d$. 
Let $X^1$ be a copy of $\mathbb{P}^2$; its hyperplane class will be denoted by $b_1$.   
Define 
\begin{align} 
\mathsf{A}_1 & \,\,\subset\,\, \mathcal{D}_d \times X^1 \nonumber
\end{align}
to be 
the locus of all $(H_d,\, p)$ such 
that the curve $H_d$ has a node at $p$. 
Let $\big[\mathsf{A}_1\big]$ denote the homology class represented by the closure of 
$\mathsf{A}_1$ in $\mathcal{D}_d \times X^1$. Note that 
$\big[\mathsf{A}_1\big]$ is a polynomial in $y_d$ and $b_1$; knowing this polynomial will 
enable us to compute all the characteristic number of nodal degree $d$-curves. 

For example, note that  
\[ \big[\mathsf{A}_1\big]\cdot y_d^{\delta_d-1} \]
is precisely equal to $\mathsf{N}_d(\mathsf{A}_1)$, 
the number of nodal degree $d$ curves in $\mathbb{P}^2$ that pass through $\delta_d-1$ generic points. 
However, knowing $\big[\mathsf{A}_1\big]$ also enables us to compute some other numbers.  
For example, 
\[ \big[\mathsf{A}_1\big]\cdot y_d^{\delta_d-2} \cdot b_1 \]
is equal to $\mathsf{N}_d(\mathsf{A}_1^{\mathsf{L}})$, the 
number of nodal degree $d$ curves in $\mathbb{P}^2$ that pass through $\delta_d-2$ generic points 
with the nodal point lying on a line.  

Instead of focussing on computing the \textit{number} 
$\mathsf{N}_d(\mathsf{A}_1)$, we will focus on computing the homology class 
represented by the closure of the cycle
$\mathsf{A}_1$. It may be mentioned that this class is actually not computed by Caporaso-Harris. Before proceeding further, let us define a few spaces that will be useful later.

Let $d$ be a positive integer and $n,m$ be non-negative integers. Define
\[
\mathsf{M}_{n}^m(d)\,\,:=\,\, \mathcal{D}_1 \times \mathcal{D}_d \times (X^1 \times \cdots \times X^m)\times (X_1\times \cdots\times X_n),
\]
where each $X^j$ or $X_i$ is $\mathbb{P}^2$. 
The hyperplane classes of $\mathcal{D}_1$, $\mathcal{D}_d$, $X^j$ and $X_i$ will be denoted by  
$y_1, y_d, b_j$ and $a_i$ respectively. For notational ease, $\mathsf{M}_{n}^m(d)$ will be abbreviated
by $\mathsf{M}_{n}^m$. 
Given a subspace $\mathsf{Z}$ of $\mathsf{M}_{n}^m$, the homology class represented by the closure of $\mathsf{Z}$
in $\mathsf{M}_{n}^m$ will be denoted by $\big[\mathsf{Z}\big]$.

Define $\mathsf{A}_1^{\mathsf{F}}\,\, \subset\, \,\mathsf{M}_{0}^1$ 
to be the locus of all $(H_1,\, H_d,\, p)$ such 
that the curve $H_d$ has a node at $p$ and  $p$ does not lie on $H_1$. Let 
\begin{align}\label{definition_A_1T_0^n}
\mathsf{A}_1^{\mathsf{F}}
\underbrace{\mathsf{T}_{0}\cdots \mathsf{T}_{0}}_{n-\textnormal{times}}\,\, \subset\, \,\mathsf{M}_{n}^1
\end{align}
be the locus of all $(H_1,\, H_d,\, p,\, x_1,\, \ldots ,\, x_n)$ such that 
\begin{itemize}
\item the curve $H_d$ has a node at $p$,

\item the curve $H_d$ intersects $H_1$ at $x_1,\, \cdots  ,\,x_n$  transversally,

\item the points $x_1, \, \cdots ,\, x_n$ are all distinct, and

\item the nodal point $p$ does not lie on $H_1$.
\end{itemize}
The space can be pictorially described as follows: 
\begin{center}
            \begin{figure}[h]
                \centering
                \includegraphics[scale=1]{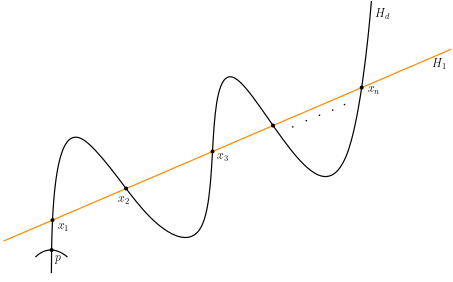}
                \caption{$\mathsf{A}_1^{\mathsf{F}}
\mathsf{T}_{0}\cdots \mathsf{T}_{0}$}
                \label{A^1FT_0}
            \end{figure}
        \end{center}

Similarly, define 
\begin{align}\label{definition_A_1T_0^nF}
\mathsf{A}_1^{\mathsf{F}} \underbrace{\mathsf{T}_{0}\cdots \mathsf{T}_{0}}_{n-\textnormal{times}} \mathsf{F}
\,\, \subset\, \,\mathsf{M}_{n+1}^1
\end{align}

to be 
the locus of all $(H_1,\, H_d,\, p,\, x_1, \ldots\, ,x_n,\, x_{n+1})$ such 
that 
\begin{itemize}
\item the curve $H_d$ has a node at $p$,

\item the curve $H_d$ intersects $H_1$ at $x_1,\, \ldots,\, x_n$  transversally,

\item the points $x_1,\, x_2, \, \ldots,\, x_n,\, x_{n+1}$ are all distinct, and

\item the nodal point $p$ does not lie on $H_1$.
\end{itemize}
The space can be pictorially described as follows: 

\begin{center}
            \begin{figure}[h]
                \centering
                \includegraphics[scale=1]{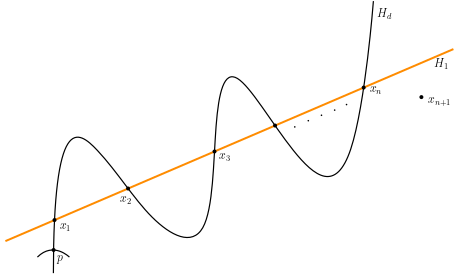}
                \caption{$\mathsf{A}_1^{\mathsf{F}}
\mathsf{T}_{0}\cdots \mathsf{T}_{0} \mathsf{F}$}
                \label{A^1FT_0F}
            \end{figure}
        \end{center}

Next, define 
\[ \big( x_i \in \mathsf{L} \big) \,\,\subset\,\, \mathsf{M}^1_n \] 
to be the locus of all $(H_1,\, H_d,\, p,\, x_1,\, \cdots,\, x_n)$ such 
that $x_i$ lies on the line $H_1$. Similarly, define 
\[ \big( x_i \in \mathcal{C}_d \big) \,\,\subset\,\, \mathsf{M}^1_n \] 
to be the locus of all $(H_1,\, H_d,\, p,\, x_1,\, \cdots,\, x_n)$ such 
that $x_i$ lies on the curve $H_d$.

Finally, for $i\,\neq\, j$, define 
\[ \big(x_i = x_j \big)\,\, \subset\, \,\mathsf{M}^1_n  \]
to be the locus 
of all $(H_1,\, H_d,\, p,\, x_1,\, \cdots,\, x_n)$ such that $x_i$ is equal to $x_j$. 
Similarly, define 
\[ \big(p = x_j \big)\,\,\, \subset\, \,\,\mathsf{M}^1_n  \]
to be the locus 
of all $(H_1,\, H_d,\, p,\, x_1,\, \cdots,\, x_n)$ such that $p$ is equal to $x_j$. 

Let us now recall a few standard facts; the following equalities of homology classes in 
$\mathsf{M}^1_n$ hold: 
\begin{align}
\big[\big(x_i = x_j \big)\big] & \,=\, a_i^2 + a_i a_j + a_j^2, \qquad 
\big[\big(p = x_j \big)\big] \,=\, b_1^2 + b_1 a_j + a_j^2, \nonumber \\ 
\big[\big( x_i \in \mathsf{L} \big)\big] & \,=\, y_1 + a_i, \qquad \textnormal{and} \qquad
\big[\big( x_i \in \mathcal{C}_d \big)\big]\, =\, y_d + d a_i. \label{diag_inc_std}
\end{align}

\subsection{Cycle version}

We are now ready to state the main results on the level of cycles. We will first state and prove all the statements on cycle level. Enumerative information are extracted  from them in \Cref{enumerative applications}.

\begin{thm} 
\label{step1_CH_p}
The following equality of homology classes hold in  
$\mathsf{M}^1_1(d)$: 
\begin{align*}
\big[\mathsf{A}_1^{\mathsf{F}} \mathsf{F}\big]\cdot
\big[\big(x_1 \in \mathsf{L}\big)\big] \cdot \big[ \big(x_1 \in \mathcal{C}_d\big) \big] 
&\,\,=\,\, \big[\mathsf{A}_1^{\mathsf{F}}\mathsf{T}_0\big], 
\end{align*}
where $d \,\in\, \mathbb{Z}^{\geq 2}$. 
\end{thm}

\Cref{step1_CH_p} is a special case of 
the following theorem when $n\,=\,0$. We will prove this directly.  

\begin{thm}
\label{theorem_for_many_Tks_A1F_pp}
Let $d \,\in\, \mathbb{Z}^{\geq 2}$, and let $n$ be an integer such that 
\[0 \ \leq\  n \ \leq\ d-2.\]
Then, on $\mathsf{M}^1_{n+1}(d)$ 
the following equality of homology classes hold: 
\begin{align}\label{many_Tks_A1F_pp}
\big[\mathsf{A}_1^{\mathsf{F}} \underbrace{\mathsf{T}_{0} \cdots \mathsf{T}_{0}}_{n-\textnormal{times}}\mathsf{F}\big]\cdot 
\big[\big(x_{n+1}\in \mathsf{L}\big)\big]\cdot \big[\big(x_{n+1}\in \mathcal{C}_d\big)\big] \,\,&= \,\,
\big[\mathsf{A}_1^{\mathsf{F}} \underbrace{\mathsf{T}_{0}\cdots \mathsf{T}_{0}}_{(n+1)-\textnormal{times}}\big]\nonumber \\ 
& + \sum_{i=1}^{n} \big[\mathsf{A}_1^{\mathsf{F}}
\underbrace{\mathsf{T}_{0}\cdots \mathsf{T}_{0}}_{n-\textnormal{times}} \mathsf{F}\big] \cdot \big[\big(x_i = x_{n+1}\big)\big].
\end{align}
\end{thm}

\begin{rem}\label{remark1}
Let us first examine the validity of \cref{many_Tks_A1F_pp} on set-theoretic level. 
Consider the first term 
on the left-hand side, namely 
\[\big[\mathsf{A}_1^{\mathsf{F}} \underbrace{\mathsf{T}_{0} \cdots \mathsf{T}_{0}}_{n-\textnormal{times}}\mathsf{F}\big].\] 
It is represented by the closure of the following space: a line, a nodal curve with the node not lying on the line 
and $n+1$ distinct points $(x_1,\,
\cdots,\, x_{n+1})$, such that the curve intersects the line transversally 
at the points $(x_1,\, \cdots,\, x_n)$. 
The last point $x_{n+1}$ is free (it does not have to lie on either the line or the curve). 

Now consider the second and third factor on the left-hand side of 
\cref{many_Tks_A1F_pp}, namely 
\[ \big[\big(x_{n+1} \in \mathsf{L}\big) \big] \cdot \big[ \big(x_{n+1} \in \mathcal{C}_d\big) \big]. \]
It is represented by the following space: a
line, a curve and $n+1$ points $(x_1,\, \cdots,\, x_{n+1})$, 
such that the points $(x_1,\, \cdots,\, x_n)$ are free, while the last point $x_{n+1}$ lies on the line
and the curve.

The set-theoretic intersection of the above two spaces has two possibilities. 
The first case to consider is when the nodal point does not lie on the line. 
Corresponding to this, there are two sub-cases. 
The first sub-case is that the point $x_{n+1}$ is distinct from all the other points $(x_1,\, \cdots,\, x_n)$. 
The closure of that space represents the first term on the right-hand side of
\cref{many_Tks_A1F_pp}.

The second sub-case is that the point $x_{n+1}$ could be equal to one of the $x_i$ 
(for $i\,\in\, \{1,\, \ldots,\, n\}$). 
The closure of that space gives us the second term on the right-hand side of \cref{many_Tks_A1F_pp}. 

The second possibility to consider is that the nodal point lies on the line. 
It will be shown that this component has 
a high codimension and hence will not intersect a cycle of complementary dimension.

First note that $\mathsf{M}^1_{n+1}(d)$ is a manifold of dimension 
$\delta_d+2n+6$.
\Cref{many_Tks_A1F_pp} is an 
equality of classes of codimension 
$2n+5$. This is because $\mathsf{A}_1^{\mathsf{F}}$ is a codimension three 
class. Furthermore, imposing the condition that the point lies on both the line and the 
curve is a codimension $2$ condition. Hence, the total codimension is 
\[ 3 + 2(n+1) = 2n+5. \] 
Therefore, the dimension of the homology cycles representing each side of \cref{many_Tks_A1F_pp} 
is given by
\begin{align*}
a&\, :=\, \delta_d+2n+6 - (2n+5) \\ 
 & \,=\, \frac{d^2}{2}+\frac{3d}{2}+1.   
\end{align*}
Now the dimension of the component where the nodal point lies on the 
line will be computed. More precisely, we are looking at the space of degree $d$ curves, with a nodal 
point on the designated line and $n+1$ points lying on the line. 
Assume further that the curve is either irreducible or even if it is reducible, the 
designated line is not one of the components of the curve (i.e., there exists points on the line that 
do not lie on the curve). The nodal point lying on 
the line is a codimension $4$ condition. Each of the $(n+1)$ marked points lying on the line and curve 
is a codimension $2$ condition. Hence, the total codimension of this space is $2n+6$.   
This is strictly greater than $2n+5$ and hence it does not contribute to the intersection. 

Now look at the space of degree $d$ curves, with a nodal 
point on the designated line and $n+1$ points lying on the line. 
Furthermore, assume that the curve is reducible, with the line being one of the components of the curve. 
The dimension of this space is given by  
\begin{align*}
b&\ :=\ 2 + \delta_{d-1} + (n+1). \\ 
 & \ =\ \frac{d^2}{2} + \frac{d}{2} + n+2.
\end{align*}   
To see why this is so, note that the space of lines is two-dimensional and the space of degree $d-1$ curves 
is of dimension $\delta_{d-1}$. We are now allowed to place $n+1$ points on the line; each point has one
degree of freedom. Hence, the total dimension of the space is $2 + \delta_{d-1} + (n+1)$.    

Note that 
\begin{align}
b-a & \ =\ n-(d-1). \label{step_d_dim_red}
\end{align} 
Hence, if $n\,<\,d-1$, then $b\,<\,a$; this implies that the 
space of reducible curves with the line being one of the components does not contribute 
to the intersection if $n\,<\,d-1$. Notice the importance of the inequality $n\,<\,d-1$ for this argument 
(as we will see in the next theorem, this component does contribute to the intersection).   
\end{rem}

\begin{proof}[\textbf{Proof of  \Cref{theorem_for_many_Tks_A1F_pp}}]
Following the discussion in \Cref{remark1}, it suffices to show that the relevant intersections 
are transverse and that the multiplicities of the intersections for the second term on the 
right-hand side of \cref{many_Tks_A1F_pp} are all one. 

We claim that the open part of each of the cycles is actually a smooth 
manifold of the expected dimension. This can be proved by showing that $\mathsf{A}_1^{\mathsf{F}}$ 
is a smooth submanifold of $\mathsf{M}^1_0$ of codimension three (i.e., it has the expected dimension).
This is a local statement, we switch to affine spaces. Let $\mathcal{F}_d$ be the space of 
polynomials in two variables of degree at most $d$. 
This is a vector space of dimension $\frac{d(d+3)}{2} + 1$, because
an element of $\mathcal{F}_d$ can be viewed as 
\[
 f(x,\,y)\,:=\, f_{00} + f_{10}x + f_{01} y + \frac{f_{20}}{2} x^2 + f_{11} xy + \frac{f_{02}}{2} y^2 +
\cdots + \frac{f_{0d}}{d!} y^d,
\]
and hence $f$ can be identified with $(f_{00},\, f_{10},\, f_{01},\, \ldots,\, f_{0d})$.
Let $\mathcal{F}_d^{*}$ denote the space of all non-zero polynomials of degree $d$.

Define $(\mathsf{A}_1^{\mathsf{F}})_{\textsf{Aff}}$ to be the following subspace of 
$\mathcal{F}_d^* \times \mathbb{C}^2$: it consists of a curve $f$ and a pair of complex numbers 
$(x,\, y) \,\in\, \mathbb{C}^2$ such that the curve 
has a node at $(x,\, y)$. 
It will be shown that $(\mathsf{A}_1^{\mathsf{F}})_{\textsf{Aff}}$ 
is a smooth
submanifold of $\mathcal{F}_d^* \times \mathbb{C}^2$.

Consider the map 
\[
\varphi\,\,:\,\,\mathcal{F}_d^* \times \mathbb{C}^2\,\, \longrightarrow\,\, \mathbb{C}^3, \ \ \, 
(f,\, (x,\, y))\,\longmapsto\,(f(x,\,y),\, f_x(x,\,y),\, f_y(x,\,y)).
\]
It suffices to show that $0$ is a regular value of $\varphi$. Consider the polynomials $\eta_{ij}(x,\,y)$ given by 
\begin{align*}
\eta_{00}(x,\,y)&\,:=\, 1, \qquad \eta_{10}(x,\,y)\,:=\, x, \qquad \textnormal{and} \qquad \eta_{01}(x,\,y)\,:=\, y. 
\end{align*}
Let $\gamma_{ij}(t)$ be the curve given by 
\[ \gamma_{ij}(t)\,\,:=\,\, (f+t \eta_{ij},\, x). \]
Computing the differential of $d\varphi$ on these curves proves transversality. 

Next, we will show that 
$\mathsf{A}_1^{\mathsf{F}} \mathsf{T}_0$ 
is a smooth submanifold of $\mathsf{M}^1_1$ of codimension five. 
Let us now fix the $\mathtt{x}-$axis as the designated line. 
Define $(\mathsf{A}_1^{\mathsf{F}}\mathsf{T}_0)_{\textsf{Aff}}$ to be the following subspace of 
$\mathcal{F}_d^* \times \mathbb{C}^2 \times \mathbb{C}$: 
it consists of a curve $f$ and a triple of complex numbers $(x,\,y,\, x_1) \,\in
\,\mathbb{C}^3$ such that $f$ has a node at $(x,\,y)$ and 
it intersects the $\mathtt{x}-$axis at $(x_1,\,0)$ transversally.
Furthermore, we have $y\, \neq\, 0$ (this corresponds to the condition that the nodal point does not lie on the line). 
It will be shown that $(\mathsf{A}_1^{\mathsf{F}}\mathsf{T}_0)_{\textsf{Aff}}$
is a smooth submanifold of $\mathcal{F}_d^* \times \mathbb{C}^3$. For that, define the map 
\[
\varphi\,:\,\mathsf{A}_1^{\mathsf{F}} \times \mathbb{C} \, \longrightarrow\, \mathbb{C}, 
\ \ \ \, (f,\, (x,y),\, x_1)\, \longmapsto\, f(x_1,\, 0).
\]
In order to show that 
$(\mathsf{A}_1^{\mathsf{F}}\mathsf{T}_0)_{\textsf{Aff}}$
is a smooth submanifold of 
$\mathsf{A}_1^{\mathsf{F}} \times \mathbb{C}$, 
it suffices to prove that $0$ is a regular value of $\varphi$. 
To show that $0$ is a regular value of $\varphi$, assume that $\varphi(f,\,(x_0,\,y_0),\,x_1)\,=\,0$. 
For computing the differential of $\varphi$ at $(f,\,(x_0,\,y_0),\,x_1)$, consider the following curve 
\[ 
\gamma(t)\,\,:=\,\, (f + t \eta, \, (x_0,\,y_0), \,x_1), 
\]
where $\eta$ is as yet an unspecified polynomial. 
We need to choose $\eta$ in such a way that 
\[\{d\varphi|_{(f,\,(x_0,y_0),\,x_1)}\}(\gamma^{\prime}(0))\] 
is non-zero. Furthermore, $f+t \eta$ should have a node at $(x_0,\,y_0)$ for all $t$. 
Since $f$ has a node at $(x_0,\,y_0)$, this is equivalent to the condition that 
\begin{align}
\eta(x_0, \,y_0)&\, =\, 0, \qquad \eta_x(x_0,\, y_0)\,=\, 0 \qquad \textnormal{and} \qquad
\eta_y(x_0,\, y_0) \,=\, 0. \label{eta_node}
\end{align} 
Now note that the differential of $\varphi$ is given by 
\[
\left\{d\varphi\big\vert_{(f,(x,y), x_1)}\right\}(\gamma^{\prime}(0)) \,\,=\,\,\,
\lim_{t\rightarrow 0} \frac{(f+t\eta)(x_1,\, 0) - f(x_1,\, 0)}{t}\,\, =\,\, \eta(x_1,\, 0). 
\]
Hence, in addition to the conditions of \cref{eta_node}, 
it is required that $\eta(x_1,\,0) \,\neq\, 0$. 
We now show that such an $\eta$ exists. Let 
\[ \eta(x,\,y)\ := \ (y-y_0)^2.  \] 
Since $y_0 \,\neq\, 0$, it follows that all the four conditions are satisfied. 
This proves that
$(\mathsf{A}_1^{\mathsf{F}}\mathsf{T}_0)_{\textsf{Aff}}$ is a smooth submanifold of 
$\mathsf{A}_1^{\mathsf{F}} \times \mathbb{C}$.

Next, we will show that 
$\mathsf{A}_1^{\mathsf{F}} \mathsf{T}_0 \mathsf{T}_0$ 
is a smooth submanifold of $\mathsf{M}^1_2$ of codimension seven. 
As before, we switch to affine space and define 
$(\mathsf{A}_1^{\mathsf{F}}\mathsf{T}_0 \mathsf{T}_0)_{\textsf{Aff}}$ as follows: 
\[
(\mathsf{A}_1^{\mathsf{F}}\mathsf{T}_0 \mathsf{T}_0)_{\textsf{Aff}} \,\,:=\,\, 
\{ (f,\, (x_0,y_0),\, x_1, \,x_2)\,\in\, (\mathsf{A}_1^{\mathsf{F}}\mathsf{T}_0)_{\textsf{Aff}}\times \mathbb{C}\,\,
\big\vert\,\,f(x_2,0)\,=\,0 ,\,\, x_1 \,\neq\, x_2\}.
\]

To show that 
$(\mathsf{A}_1^{\mathsf{F}}\mathsf{T}_0 \mathsf{T}_0)_{\textsf{Aff}}$
is a smooth submanifold of 
$(\mathsf{A}_1^{\mathsf{F}}\mathsf{T}_0)_{\textsf{Aff}}\times \mathbb{C}$, define the map 
\[
\varphi\,:\,(\mathsf{A}_1^{\mathsf{F}}\mathsf{T}_0)_{\textsf{Aff}}\times \mathbb{C} \, \longrightarrow\, \mathbb{C}, 
\ \ \, (f,\,(x_0,\, y_0),\, x_1, \,x_2)\, \longmapsto\, f(x_2,\,0).
\]
It suffices to show that the differential of $\varphi$ is nonzero when $x_1 \,\neq\, x_2$ and when $y_0
\,\neq\, 0$. 
To compute the differential of $\varphi$ at $(f,\,(x_0,\, y_0),\,x_1,\, x_2)$, consider the curve 
\[
\gamma(t)\,\,:=\,\, (f + t \eta, \,(x_0,\, y_0), \,x_1, \,x_2), 
\]
where $\eta$ is as yet an unspecified polynomial. It should be so that 
\[\{d\varphi|_{(f,x_1, x_2)}\}(\gamma^{\prime}(0))\] 
is non-zero. Furthermore, the curve $\gamma(t)$ is required to lie in 
$(\mathsf{A}_1^{\mathsf{F}}\mathsf{T}_0)_{\textsf{Aff}} \times \mathbb{C}$ for all $t$. 
This means that \cref{eta_node} must be satisfied and furthermore  
\begin{align}
\eta(x_1, 0) & \ =\ 0. \label{eta_T0}
\end{align}
The differential is given by 
\[
\{d\varphi|_{(f,(x_0, y_0), x_1, x_2)}\}(\gamma^{\prime}(0)) \,\, =\,\,
\lim_{t\rightarrow 0} \frac{(f+t\eta)(x_2,\, 0) -f(x_2,\, 0)}{t}\,\, =\,\, \eta(x_2,\, 0).
\]
Hence, $\eta$ should satisfy \cref{eta_node,eta_T0}, 
and 
$\eta(x_2,\, 0) \,\neq\, 0$. 
Such an $\eta$ exists; for example consider 
\[
\eta(x,\,y)\,\,:=\,\, (y-y_0)^2(x-x_1).
\]
Notice that the conditions $y_0 \,\neq \,0$ and $x_1 \,\neq\, x_2$ are crucial.
We also need $d \,\geq\, 3$ for the argument to work (since the polynomial $\eta$ is of degree $3$). 

Using a similar argument, it is seen that 
\[\big[\mathsf{A}_1^{\mathsf{F}} \underbrace{\mathsf{T}_{0} \cdots \mathsf{T}_{0}}_{n-\textnormal{times}}\big]\]
is a smooth submanifold of $\mathsf{M}^1_{n}$ of codimension $3+2n$, provided $n \,\leq\, d-2$. 
The bound on $d$ enables one to construct the desired curve $\eta$ that is needed to prove the transversality.  

Now the multiplicities of the intersections will be computed. 
The above transversality shows that the first 
term in the right-hand side of \cref{many_Tks_A1F_pp} 
occurs with multiplicity one.
To justify the intersection multiplicities of the remaining terms, consider the situation where
$x_1$ coincides with $x_{n+1}$. For convenience, set $x_{n+1}$ to be equal to zero 
(i.e., the last point is the origin); in that case, $f$ intersects the $\mathtt{x}$--axis transversally
at the origin. We are now going to 
study the multiplicity with which the evaluation map vanishes at the origin. 
Hence, $f$ is such that $f_{00}$ vanishes. It is given by 
\[
f(x,\,y)\,\, =\,\, f_{10} x + f_{01} y + \frac{f_{20}}{2} x^2 + f_{11} xy + \frac{f_{02}}{2} y^2 + \cdots .
\]
Now consider the evaluation map 
\[
\varphi(f,\, x)\,\,:=\,\, f(x,\,0) \,\, =\,\,f_{10} x + \frac{f_{20}}{2} x^2 + \cdots. 
\]
The order of vanishing of $\varphi$ is clearly $1$, provided 
$f_{10} \,\neq\, 0$. 
But that assumption is valid, since $f$ intersects the $\mathtt{x}$--axis transversally at the origin.  
Hence, the order of vanishing is $1$. This proves the equality in \cref{many_Tks_A1F_pp}. 
\end{proof}

\Cref{theorem_for_many_Tks_A1F_pp} is false, when $n\,=\,d-1$. 
Before stating the correct statement for $n\,=\, d-1$,  
we need to develop some more notation. Define 
\begin{align}\label{definition_R_n}
\mathsf{R}_n(d)\ \subset\ \mathsf{M}^1_{n}(d) 
\end{align} 
to be the locus of all $(H_1,\, H_d,\, p,\, x_1,\, \cdots,\, x_n)$ such that
\begin{itemize}
\item the curve $H_d$ is the union of the line $H_1$ and a degree $d-1$ curve,  
\item the degree $d-1$ curve is smooth, 
\item the degree $d-1$ curve intersects $H_1$ transversally at all the points. 
\item the point $p$ is one of the points of intersection of the $d-1$ curve and $H_1$, and 
\item the points $x_1, \, \cdots,\, x_n$ are all distinct and lie on $H_1$; they are also distinct from $p$. 
\end{itemize}
The space can be pictorially described as follows: 

\begin{center}
            \begin{figure}[h]
                \centering
                \includegraphics[scale=1]{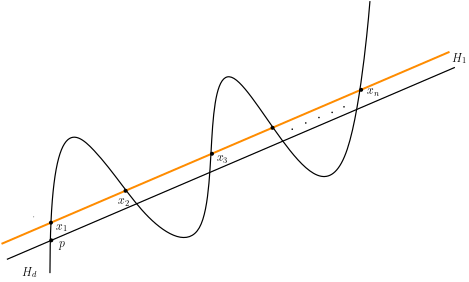}
                \caption{$\mathsf{R}_n(d)$}
                \label{R_nd}
            \end{figure}
        \end{center}

\begin{rem} \label{black_orange}
    Note that the black and orange lines in figure \ref{R_nd} are actually supposed to represent the same line. We distinguish them to keep in mind that the line is a part of a reducible curve, as well as a line in its own moduli space. This remark also applies to the rest of the pictures in this section.
\end{rem}

Insertion of one more point on both the line and the curve would yield the following result.

\begin{thm}
\label{theorem_for_many_Tks_A1F_pp_Break}
Let $d \in \mathbb{Z}^{\geq 3}$. 
Then, the following equality of homology classes 
holds in $\mathsf{M}_{{d}}^1(d)$
\begin{align}\label{many_Tks_A1F_pp_Break}
\big[\mathsf{A}_1^{\mathsf{F}} \underbrace{\mathsf{T}_{0} \cdots \mathsf{T}_{0}}_{(d-1)-\textnormal{times}} \mathsf{F}\big]\cdot 
\big[\big(x_d \in \mathsf{L}\big)\big]\cdot \big[\big(x_d \in \mathcal{C}_d\big)\big] \,\,&= \,\,
\big[\mathsf{A}_1^{\mathsf{F}} \underbrace{\mathsf{T}_{0}\cdots \mathsf{T}_{0}}_{d-\textnormal{times}}\big] \nonumber \\ 
& + \sum_{i=1}^{d-1}\big[\mathsf{A}_1^{\mathsf{F}}
\underbrace{\mathsf{T}_{0}\cdots \mathsf{T}_{0}}_{(d-1)-\textnormal{times}}\mathsf{F}\big] \cdot \big[\big(x_i = x_d\big)\big] 
+ \big[\mathsf{R}_d(d)\big].
\end{align}
\end{thm}
\begin{rem}
Let us first analyze the following set theoretic question: consider the component of 
the closure of 
\[\mathsf{A}_1^{\mathsf{F}} \underbrace{\mathsf{T}_0 \cdots \mathsf{T}_0}_{(d-1)-\textnormal{times}}\]
when the nodal point $p$ lies on the line. If the nodal point is equal to one of the other points $x_1,\,
x_2,\, \cdots,\, x_{d-1}$, 
then a dimension counting argument shows that this component will not contribute to the intersection. 
Now assume that this point is different from all the other marked points. This is possible only
if the curve is reducible and the line is one of the components 
of the curve. To see why this is so, note that if the line is not one of the components, then the nodal 
point contributes at least $2$ to the intersection and there are $d-1$ other distinct points. The curve 
intersects the line at $d+1$ points; this is impossible, since the curve has degree $d$. 

Hence the only configuration that can potentially arise (that contributes to the 
intersection) is a reducible curve, with one of the components the line and the remaining points $x_1,\, x_2,\, 
\cdots,\, x_{d-1}$ lying on the line, with the nodal point being distinct from all of them. Recall from 
\cref{step_d_dim_red} that since $n$ is now equal to $d-1$, dimensional considerations do not rule 
out this configuration. \Cref{theorem_for_many_Tks_A1F_pp_Break} states that this configuration 
and $x_{d}$ coincides with any of the $x_i$'s are the only possibility that can occur in the degeneration.
\end{rem}

\begin{proof}[\textbf{Proof of \Cref{theorem_for_many_Tks_A1F_pp_Break}}]
We continue with the setup 
of \Cref{theorem_for_many_Tks_A1F_pp}. The first and second terms on the right-hand side 
are justified as before. Let us  now justify the third term, which is the new thing that occurs. 
It will be shown that this configuration actually 
arises. In other words, we need to show that 
\begin{align}
\overline{\mathsf{R}}_d(d) & 
\subset 
\overline{\mathsf{A}_1^{\mathsf{F}} \underbrace{\mathsf{T}_{0} \cdots \mathsf{T}_{0}}_{(d-1)-\textnormal{times}} \mathsf{F}}. \label{k_clsr}
\end{align}
It suffices to show that 
\begin{align}
\mathsf{R}_d(d) & 
\subset 
\overline{\mathsf{A}_1^{\mathsf{F}} \underbrace{\mathsf{T}_{0} \cdots \mathsf{T}_{0}}_{(d-1)-\textnormal{times}} \mathsf{F}}. \label{k_clsr3}
\end{align}
Let $(H_1,\, H_d,\, p,\, x_1,\, \cdots,\, x_{d-1},\, x_d)$ be an arbitrary point of 
$\mathsf{R}_d(d)$. 
We will switch to affine space as before. Assume the line $H_1$ is the $\mathtt{x}$--axis and the nodal point
$p$ is the origin $(0,\,0)$. Also assume that 
$x_1,\, x_2,\, \cdots,\, x_{d-1}$ are given by $(a_1,\, 0)$, $(a_2,\, 0),\, \cdots,\, (a_{d-1},\,0)$. 
Since these points are all distinct and also not equal to the nodal point, it follows that
$a_1,\, a_2,\, \cdots,\, a_{d-1}$ and $0$ are all distinct.  

Since $H_d$ is a reducible curve, with the line $H_1$ being one of the reducible components, 
it is given in the affine coordinate by a polynomial $f$ of the form
\begin{align*}
f(x,y)& \ =\ y \varphi(x,y),
\end{align*} 
where $\varphi(x,\,y)$ is a polynomial of degree $d-1$. Let us write down $\varphi(x,\,y)$ explicitly as 
\begin{align}
\varphi(x,y)&\ :=\ b_{00} + b_{10} x + b_{01} y + \cdots +\frac{b_{0, d-1}}{(d-1)!} y^{d-1}. \nonumber
\end{align}
Recall that the origin is a nodal point of $f$. 
By definition, $\phi$ passes through the origin and intersect the $\mathtt{x}$--axis transversally
at origin. Hence, $b_{00} \,=\,0$ and $b_{10} \,\neq\, 0$. 

Therefore, we have
\begin{align}
\Big(f,\, (0,\,0),\, a_1,\, \cdots,\, a_{d-1},\, a_d\Big) & \,\in\,
\big(\mathsf{R}_d(d)\big)_{\mathsf{Aff}}\,\subset \, \mathcal{F}_d^* \times \mathbb{C}^2 \times \mathbb{C}^d. 
\nonumber
\end{align}
To prove \cref{k_clsr3}, it suffices to show the following: \\
given any sufficiently small open set $U$ of $\mathcal{F}_d^* \times \mathbb{C}^2 \times \mathbb{C}^d$
containing $\left(f,\, (0,\,0),\, a_1,\, \cdots,\, a_{d-1},\, a_d\right)$,
there exists an element 
\[\Big(F,\, (s,\,t),\, A_1,\, \cdots,\, A_{d-1},\, A_d\Big)\ \in\ U\] 
such that
\begin{align}
\Big(F,\, (s,\,t),\, A_1,\, \cdots,\, A_{d-1},\, A_d\Big) \ \in\ 
\Big(\mathsf{A}_1^{\mathsf{F}} \underbrace{\mathsf{T}_{0} \cdots
\mathsf{T}_{0}}_{(d-1)-\textnormal{times}} \mathsf{F}\Big)_{\mathsf{Aff}}. 
\label{CH_step3_F_const}
\end{align}

This element in \eqref{CH_step3_F_const} will be constructed explicitly. First of all, note that the curve
given by $F$, if exists, passes through the points $(A_1,\, 0),\,\cdots,\, (A_{d-1},\, 0)$. Hence, $F$ 
has to be of the form
\begin{align*}
F(x,y)&\ =\ (x-A_1) \cdots (x-A_{d-1}) (C_0 + C_1 x) + y \Phi(x,y),
\end{align*}
where $\Phi$ is a polynomial of degree $d-1$. 
Write down $\Phi$ explicitly as 
\begin{align*} 
\Phi(x,y) &\ :=\ B_{00} + B_{10} x + B_{01} y + \cdots + \frac{B_{0, d-1}}{(d-1)!} y^{d-1}.
\end{align*}
Since $U$ is a sufficiently small open neighbourhood of 
$\left(f,\, (0,\,0),\, a_1,\, \cdots,\, a_{d-1},\, a_d\right)$, 
it follows that $C_0,\, C_1$ and $B_{00}$ are small and $B_{10}$ is non-zero. 
Furthermore, $A_i$ are all close to $a_i$. Hence, $A_1,\, \cdots,\, A_{d-1}$ 
are all distinct and different from $0$. 
We now wish to find a point $(s,\,t)$ close to the origin, such that $t\,\neq\, 0$ (i.e., the point
$(s,\,t)$ does not lie on 
the $\mathtt{x}$--axis) 
and is a nodal point of $F$. Hence, it is needed to solve the following set of equations 
\begin{align}
F(s,t) & \,=\, 0, \qquad F_x(s,t) \,=\, 0 \qquad \textnormal{and} \qquad F_y(s,t) \,=\, 0, \nonumber \\   
\qquad \textnormal{where} \qquad t & \,\neq\, 0 \qquad \textnormal{and} \qquad (s,\,t) \quad \textnormal{is small}.   \label{ch_step3_node_eqn}
\end{align}  
To make the computations more convenient, rewrite $F(x,\,y)$ as follows: 
\begin{align}
F(x,y)& \,=\, Q(x) (C_0 + C_1 x) + y (B_{00} + P(x,y)), \qquad \textnormal{where} \nonumber \\ 
Q(x) & \,:=\, (x-A_1) \cdots (x-A_{d-1}), \qquad \textnormal{and} \nonumber \\  
P(x,y)&\,:=\, B_{10} x + B_{01} y + \cdots + \frac{B_{0, d-1}}{(d-1)!} y^{d-1}. \label{F_rewrite}
\end{align}
Using \cref{ch_step3_node_eqn}, 
we can solve for $C_0,\, C_1$ and $B_{00}$, which produce  
\begin{align}
B_{00}&\,=\,-P(s,t)-t P_y(s,t), \nonumber \\ 
C_0&\,=\, -\frac{-t^2Q(s) P_y(s,t) - s t^2 Q^{\prime}(s) P_y(s,t)-stQ(s)P_x(s,t)}{Q(s)^2}, \qquad \textnormal{and} \nonumber \\ 
C_1&\,=\, -\frac{t^2Q^{\prime}(s)P_y(s,t) +t Q(s) P_x(s,t)}{Q(s)^2}. \label{sol_ch_step3}
\end{align}
Since the $A_i$ are all distinct from zero, it follows  that $Q(0) \,\neq\, 0$. 
Since $s$ is close to $0$, $Q(s) \,\neq\, 0$.  
Hence, we are able to construct the 
desired solution given by \cref{sol_ch_step3}. This proves \cref{k_clsr3}, and 
hence proves \cref{k_clsr}.  

Next, to compute the multiplicity of the intersection, first note that \textit{every} curve close to 
$f$ is given by \cref{sol_ch_step3}. To compute the multiplicity of the intersection, it is necessary to    
compute $F(a_d,\, 0)$, where $a_d$ is distinct from 
$0,\, a_1,\, \cdots,\, a_{d-1}$. Using \cref{F_rewrite,sol_ch_step3}, it follows that 
\begin{align}
F(a_d, 0) & \ =\ Q(a_d) (C_0 + C_1 a_d) \nonumber \\ 
          &\ =\ \-\frac{Q(a_d) a_d B_{10}}{Q(0)} t + O(t^2). \label{mult_CH_step3}
\end{align} 
Since $a_d$ is distinct from the other $a_i$, it follows (using \cref{F_rewrite} and
the fact that $A_i$'s are close to $a_i$) that $Q(a_d)$ is non-zero. 
Furthermore, $a_d \,\neq\, 0$, since $(a_d,\, 0)$ is not the nodal point of the curve (which is the origin). 
Since $Q(a_d) \,\neq\, 0$, $a_d \,\neq\, 0$, $B_{10} \,\neq\, 0$ and $Q(0) \,\neq\, 0$, we conclude 
from \cref{mult_CH_step3} that the multiplicity of the intersection is $1$. 
The multiplicity computation completes the 
proof of \Cref{theorem_for_many_Tks_A1F_pp_Break}. 
\end{proof}

The statement of \Cref{theorem_for_many_Tks_A1F_pp_Break} is also false when $n\,=\,d$. Before
stating the correct statement for $n\,=\,d$, we need to develop yet more notation. Define
\begin{align}\label{definition_RT_1}
\big(\mathsf{R}\mathsf{T}_1\big)_n(d) \ \subset\ \mathsf{M}^1_{n}(d)
\end{align} 
to be the locus of all $(H_1,\, H_d,\, p,\, x_1,\, \cdots,\, x_{n})$ such 
that
\begin{itemize}
\item the curve $H_d$ is the union of the line $H_1$ and a degree $d-1$ curve, 
\item the degree $d-1$ curve is smooth, 
\item the degree $d-1$ curve intersects $H_1$ tangentially to first order at the point $p$; 
the remaining points of intersection are 
transverse, and
\item the points $x_1, \, \cdots, \, x_{n}$ are all distinct and lie on $H_1$; they are also distinct from $p$. 
\end{itemize}
The space can be pictorially described as follows: 

\begin{center}
            \begin{figure}[h]
                \centering
                \includegraphics[scale=1]{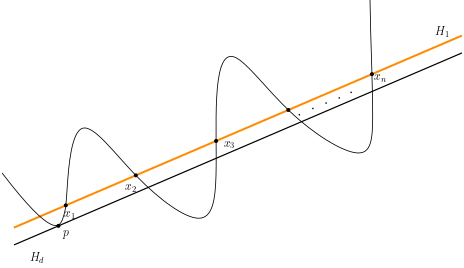}
                \caption{$\big(\mathsf{R}\mathsf{T}_1\big)_n(d)$}
                \label{RT_1nd}
            \end{figure}
        \end{center}

Also define 
\begin{align}\label{definition_RA_1}
 \big(\mathsf{R}\mathsf{A}_1^{\mathsf{F}}\big)_n(d)\ \subset\ \mathsf{M}^1_{n}(d)
\end{align}
to be 
the locus of all $(H_1,\, H_d,\, p,\, x_1,\, \cdots,\, x_{n})$ such that
\begin{itemize}
\item the curve $H_d$ is the union of the line $H_1$ and a degree $d-1$ curve, 
\item the degree $d-1$ curve is has a nodal point at $p$,
\item the nodal point $p$ does not lie on $H_1$, 
\item the degree $d-1$ curve intersects $H_1$ transversally, and
\item the points $x_1, \, \cdots,\, x_{n}$ are all distinct and lie on $H_1$. 
\end{itemize}
The space can be pictorially described as follows: 

\begin{center}
            \begin{figure}[h]
                \centering
                \includegraphics[scale=1]{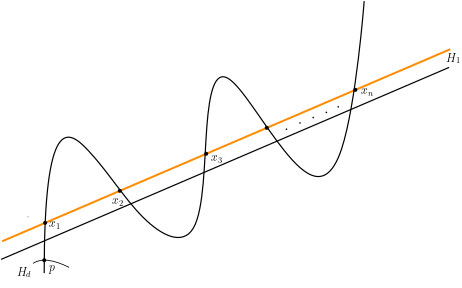}
                \caption{$\big(\mathsf{R}\mathsf{A}_1^{\mathsf{F}}\big)_n(d)$}
                \label{RA_1^Fnd}
            \end{figure}
        \end{center}

Finally, define 
\begin{align}\label{definition_R^i_n}
 \mathsf{R}_{n}^{i}(d)\ \subset\ \mathsf{M}^1_{n}(d)
\end{align}
to be the locus of all $(H_1,\, H_d,\, p,\, x_1,\, \cdots,\, x_n)$ such that
\begin{itemize}
\item the curve $H_d$ is the union of $H_1$ and a degree $d-1$ curve,   
\item the degree $d-1$ curve is smooth, 
\item the degree $d-1$ curve intersects $H_1$ transversally at all the points,
\item the point $p$ is one of the points of intersection of the $d-1$ curve and $H_1$,
\item the points $x_1, \, \cdots ,\, x_n$ are all distinct and lie on $H_1$, and
\item The point $p$ is equal to $x_i$.
\end{itemize}
The space can be pictorially described as follows: 

\begin{center}
            \begin{figure}[h]
                \centering
                \includegraphics[scale=1]{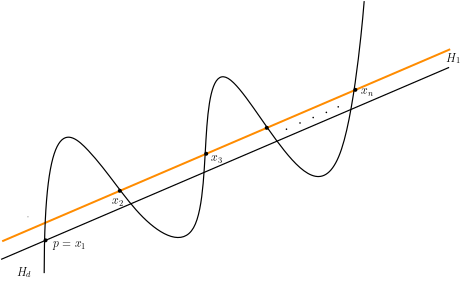}
                \caption{$\mathsf{R}_{n}^{1}(d)$}
                \label{Rn1d}
            \end{figure}
        \end{center}

\begin{thm}\label{theorem_for_many_Tks_A1F_pp_Break_Fin}
Let $d\,\in\, \mathbb{Z}^{\geq 3}$. Then, 
the following equality of homology classes in $\mathsf{M}_{{d+1}}^1(d)$ holds: 
\begin{align}\label{many_Tks_A1F_pp_Break_Fin}
\big[\mathsf{A}_1^{\mathsf{F}} \underbrace{\mathsf{T}_{0} \cdots \mathsf{T}_{0}}_{d-\textnormal{times}}\mathsf{F}\big]\cdot 
\big[\big(x_{d+1} \in \mathsf{L}\big)\big]\cdot \big[\big(x_{d+1} \in \mathcal{C}_d\big)\big] \,\,&= \,\,
\sum_{i=1}^{d}\big[\mathsf{A}_1^{\mathsf{F}}
\underbrace{\mathsf{T}_{0}\cdots \mathsf{T}_{0}}_{d-\textnormal{times}} \mathsf{F}\big] \cdot \big[\big(x_i = x_{d+1}\big)\big] \nonumber \\
& + \sum_{i=1}^d \big[\mathsf{R}^i_{d+1}(d)\big] \nonumber \\ 
& + \big[\big(\mathsf{R}\mathsf{A}_1^{\mathsf{F}}\big)_{d+1}(d)\big] + 
2 \big[\big(\mathsf{R}\mathsf{T}_1\big)_{d+1}(d)\big].
\end{align}
\end{thm}

\begin{proof}[{\textbf{Proof of Theorem \ref{theorem_for_many_Tks_A1F_pp_Break_Fin}}}]
We continue with the set-up 
of \Cref{theorem_for_many_Tks_A1F_pp} and  \Cref{theorem_for_many_Tks_A1F_pp_Break}. 
The first term on the right-hand side of \cref{many_Tks_A1F_pp_Break_Fin}
is 
justified as before. 
Let us  now justify the second, third and fourth terms. 
Let us first analyze the following situation: consider the component of 
the closure of 
\[\mathsf{A}_1^{\mathsf{F}} \underbrace{\mathsf{T}_0 \cdots \mathsf{T}_0}_{d-\textnormal{times}}\]
when the nodal point $p$ lies on the line. 
Unlike before, it is 
the component of the closure where 
the nodal point is equal to one of the other points $x_1, x_2, \ldots, x_{d}$ 
will actually contribute to the intersection. This precisely corresponds to the 
second term on the right-hand side of \cref{many_Tks_A1F_pp_Break_Fin}. 

Next, let us analyze the component of the closure  
\[\mathsf{A}_1^{\mathsf{F}} \underbrace{\mathsf{T}_0 \cdots \mathsf{T}_0}_{d-\textnormal{times}}\]
when the nodal point $p$ lies on the line, but is not equal to  
any of the other points $x_1,\, x_2,\, \cdots,\, x_{d}$. 
In such a case, it will be shown that, in the closure, the line is actually going to be tangent to one
of the branches of the node. This corresponds to the third term on  
the right hand side of \cref{many_Tks_A1F_pp_Break_Fin}; as it will be shown, it contributes with a 
multiplicity of $2$.

The component of the closure  
\[\mathsf{A}_1^{\mathsf{F}} \underbrace{\mathsf{T}_0 \cdots \mathsf{T}_0}_{d-\textnormal{times}}\]
when the nodal point $p$ does not lie on the line corresponds to the second term  
the right-hand side of \cref{many_Tks_A1F_pp_Break_Fin}. 

For the closure and multiplicity statements, we switch to affine spaces as usual. First make a general 
convention that works for all the spaces involved here.

\begin{convention}\label{convention}
Given any point $$(H_1,\, H_d,\, p,\, x_1,\, \cdots,\, x_{d},\, x_{d+1})
\,\in \,\overline{\mathsf{A}_1^{\mathsf{F}} \underbrace{\mathsf{T}_{0}
\cdots \mathsf{T}_{0}}_{d-\textnormal{times}} \mathsf{F}} \cap (x_{d+1} \in L)$$ in the affine
coordinate, the line $H_1$ is taken to be the $\mathtt{x}$--axis, the nodal point $p$ is assumed to be
the point $(a,\,b)$, and the induced element in the affine space is denoted by
$$\Big(f,\, (a,\,b),\, a_1,\, \cdots,\, a_{d},\, a_{d+1}\Big)
\,\in\, \mathcal{F}_d^* \times \mathbb{C}^2 \times \mathbb{C}^{d+1}. $$
Here, $x_1,\, x_2,\, \cdots,\, x_{d}$ are given by $(a_1,\, 0)$, $(a_2,\, 0),\, \cdots,\, (a_{d},\,0)$. 
Since these points are all distinct, it follows that $a_1,\, a_2,\, \cdots,\, a_{d}$ are all distinct.
In case, $p$ lines on $H_1$, tactically it is assumed to be the origin.
\end{convention}

Consider the second term on the 
right hand side of \cref{many_Tks_A1F_pp_Break_Fin}. 
For the closure statement, we need to show that for all $1 \,\leq\, i \,\leq\, d$,
\begin{align}
\overline{\mathsf{R}_{d+1}^i(d)} & 
\ \subset \
\overline{\mathsf{A}_1^{\mathsf{F}} \underbrace{\mathsf{T}_{0} \cdots \mathsf{T}_{0}}_{d-\textnormal{times}}
\mathsf{F}}. \label{k_clsr_ag}
\end{align}
As before, we observe that it suffices to show that 
\begin{align}
\mathsf{R}_{d+1}^i(d) & 
\ \subset \
\overline{\mathsf{A}_1^{\mathsf{F}} \underbrace{\mathsf{T}_{0} \cdots \mathsf{T}_{0}}_{d-\textnormal{times}} \mathsf{F}}. \label{k_clsr3_ag}
\end{align}

It is enough to consider the case $i\,=\,d$; other assertions are symmetric.
Let $(H_1,\, H_d,\, p,\, x_1,\, \cdots,\, x_{d},\, x_{d+1})$ be an arbitrary point of 
$\mathsf{R}_{d+1}^i(d)$. Switching to the affine spaces, the above convention will be followed.
Since $x_d$ is equal to the nodal point, it follows that $a_d \,=\,0$.

Since $H_d$ is a reducible curve, with the line $H_1$ being one of the reducible components, 
it is given by the zero set of a polynomial $f$, which is of the type 
\begin{align*}
f(x,y)& \ =\ y \varphi(x,\,y),
\end{align*} 
where $\varphi(x,\,y)$ is a polynomial of degree $d-1$. 
Write down $\varphi(x,y)$ explicitly as 
\begin{align*}
\varphi(x,y)&\ :=\ b_{00} + b_{10} x + b_{01} y + \cdots + \frac{b_{0, d-1}}{(d-1)!} y^{d-1}. 
\end{align*}
Recall that the origin is a nodal point of $f$. The degree $d-1$ component passes through the origin
and intersect the $\mathtt{x}$--axis transversally at origin. Hence, $b_{00} \,=\,0$ and $b_{10} \,\neq\, 0$. 

Now note that 
\begin{align} 
\Big(f,\, (0,\,0),\, a_1,\, \cdots,\, a_{d},\, a_{d+1}\Big) & \,\in\, 
\Big(\mathsf{R}_{d+1}^i(d)\Big)_{\mathsf{Aff}}  
\,\subset\, \mathcal{F}_d^* \times \mathbb{C}^2 \times \mathbb{C}^{d+1}. \nonumber  
\end{align}
To prove \cref{k_clsr3_ag} (equivalently, \cref{k_clsr_ag}), it suffices to show the following: given any
sufficiently small open set $U$ of $\mathcal{F}_d^* \times \mathbb{C}^2 \times \mathbb{C}^{d+1}$ containing
$\left(f,\, (0,\,0),\, a_1,\, \cdots,\, a_{d},\, a_{d+1}\right) $,
there exists an element
\begin{align*}
\Big(F,\, (s,\,t),\, A_1,\, \cdots,\, A_{d},\, A_{d+1}\Big) \,\in\, 
\Big(\mathsf{A}_1^{\mathsf{F}} \underbrace{\mathsf{T}_{0} \cdots \mathsf{T}_{0}}_{d-\textnormal{times}} \mathsf{F}\Big)_{\mathsf{Aff}} \cap U.
\end{align*}

This is shown using the idea of the proof of \Cref{theorem_for_many_Tks_A1F_pp_Break}.
Since the curve $F$ passes through the points $(A_1,\, 0),\, \cdots,\, (A_{d-1},\, 0)$ and
$(A_{d},\, 0)$, it is of the form
\begin{align*}
F(x,\,y)&\, =\, (x-A_1) \cdots (x-A_{d-1}) (x-A_d) C_0 + y \Phi(x,\,y),
\end{align*}
where $\Phi$ is a polynomial of degree $d-1$. Write $\Phi$ explicitly as 
\begin{align*} 
\Phi(x,y) &\,:=\, B_{00} + B_{10} x + B_{01} y + \cdots + \frac{B_{0, d-1}}{(d-1)!} y^{d-1}.
\end{align*}
Since $U$ is a sufficiently small open neighbourhood of $ \left(f,\, \left(0,\,0\right),\, a_1,\, \cdots,
\,a_{d-1},\, a_d,\, a_{d+1}\right)$, it is deduced that $C_0$, $A_d$ and $B_{00}$ are small and $B_{10}$ is non-zero 
(recall that $a_d\,=\,0$). Furthermore, 
$A_1,\, \cdots,\, A_{d-1}$ are all distinct and different from $0$. 
We now wish to find a point $(s,\,t)$ close to the origin, such that $t\,\neq\, 0$ (that is, the point $(s,\,
t)$ does not lie on the $\mathtt{x}$--axis) 
and is a nodal point of $F$. Hence, it is necessary to solve the following set of equations 
\begin{align}
F(s,\,t) & \,=\, 0, \qquad F_x(s,\,t) \,=\, 0 \qquad \textnormal{and} \qquad F_y(s,\,t) \,=\, 0, \nonumber \\   
\qquad \textnormal{where} \qquad t & \,\neq\, 0 \qquad \textnormal{and} \qquad (s,\,t)
\quad \textnormal{is small}. \label{ch_step3_node_eqn_ag}
\end{align}   
To make the computations more convenient, rewrite $F(x,\,y)$ as follows: 
\begin{align} 
F(x,\,y)& \,=\, Q(x) (x-A_d) C_0 + y (B_{00} + P(x,\,y)), \nonumber \\ 
\qquad \textnormal{where} \quad &Q(x)\,:=\, (x-A_1) \cdots (x-A_{d-1}),  \nonumber \\  
\qquad \textnormal{and }&P(x,\,y)\,:=\, B_{10} x + B_{01} y + \cdots + \frac{B_{0, d-1}}{(d-1)!} y^{d-1}.
\label{F_rewrite_ag}
\end{align}
Using \cref{ch_step3_node_eqn_ag}, solve for $C_0$, $A_d$ and $B_{00}$, and conclude that
\begin{align}
B_{00}&\,=\,-P(s,\,t)-t P_y(s,\,t), \nonumber \\ 
C_0&\,=\, -\frac{t^2Q^{\prime}(s)P_y(s,\,t) + tQ(s) P_x(s,\,t)}{Q(s)^2}, \nonumber \\ 
A_d & \,=\, s + \frac{B_{01}} {B_{10}} t + O(t^2). \label{sol_ch_step3_ag}
\end{align}
Note that $B_{00},\, C_0$ and $A_d$ are indeed small (i.e., they go to zero as $(s,\,t)$ goes to $(0,\,0)$). 
Hence, we have constructed solutions to \cref{F_rewrite_ag}, which proves  \cref{k_clsr_ag}.

To compute the multiplicity of the intersection, first note that \textit{every} curve close to 
$f$ is given by \cref{sol_ch_step3_ag}. In order to compute the multiplicity of the intersection, we 
simply need to compute $F(a_{d+1},\, 0)$, where $a_{d+1}$ is distinct from 
$0,\, a_1,\, \cdots,\, a_{d-1}$. Hence, using \cref{F_rewrite_ag,sol_ch_step3_ag} it follows that 
\begin{align}
F(a_{d+1}, 0) & \,=\, Q(a_{d+1})(a_{d+1}-A_d)C_0 \nonumber \\ 
          & =\, Q(a_{d+1})(a_{d+1}-A_d) \frac{B_{10}}{Q(0)} t + O(t^2). \label{mult_CH_step3_ag}
\end{align} 
Since $a_{d+1}$ is distinct from the other $a_i$, it follows using \cref{F_rewrite_ag}
(and using the fact that $A_i$'s are close to $a_i$) that $Q(a_{d+1})$ is non-zero. 
Furthermore, $a_{d+1} \,\neq\, 0$, because $(a_{d+1},\, 0)$ is not the nodal point of the curve (which is
the origin). Since 
\[Q(a_{d+1}) \,\neq\, 0, \qquad \big(a_{d+1}-A_d\big) \,\neq\,0, 
\qquad  B_{10} \,\neq\, 0 \qquad  \textnormal{and} \qquad  
Q(0) \,\neq\, 0,\] 
it follows from \cref{mult_CH_step3_ag} that the 
multiplicity of the intersection is $1$. 
The multiplicity computation completes the justification for 
the second term on the 
right-hand side of \cref{many_Tks_A1F_pp_Break_Fin} (on the level of cycles). 

Next, the fourth term on the right-hand side of \cref{many_Tks_A1F_pp_Break_Fin} will be justified.
As before, it suffices to show that 
\begin{align}
\big(\mathsf{R}\mathsf{T}_1\big)_{d+1}(d) & 
\ \subset\ 
\overline{\mathsf{A}_1^{\mathsf{F}} \underbrace{\mathsf{T}_{0} \cdots \mathsf{T}_{0}}_{d-\textnormal{times}} \mathsf{F}}. \label{k_clsr2_ag2}
\end{align}
Let $(H_1,\, H_d,\, p,\, x_1,\, \cdots,\, x_{d},\, x_{d+1})$ be an arbitrary point of 
$\big(\mathsf{R}\mathsf{T}_1\big)_{d+1}(d)$. As before, switch to affine spaces. Following \Cref{convention},
assume that $p$ --- the point of tangency of $f$ --- is the origin. Since $H_d$ is a reducible curve, with
the line $H_1$ being one of the reducible components, it follows that 
$H_d$ is given by the zero set of a polynomial $f$, which is of the type 
\begin{align*}
f(x,y)& \ =\ y \varphi(x,y),
\end{align*} 
where $\varphi(x,\,y)$ is a polynomial of degree $d-1$. Write $\varphi(x,\,y)$ explicitly as 
\begin{align*}
\varphi(x,\,y)&\,:=\, b_{00} + b_{10} x + b_{01} y + \cdots +\frac{b_{0, d-1}}{(d-1)!} y^{d-1}. 
\end{align*}
Since the degree $d-1$ component is tangent of first order to the line 
at the origin, it follows that $b_{00}\,=\,0$,  $b_{10}\,=\,0$ and $b_{20} \,\neq\, 0$. 
Furthermore, that the point is a smooth point of the degree $d-1$ curve implies that $b_{01} \,\neq\, 0$. 

Now note that 
\begin{align*} 
\Big(f, (0,0), a_1, \ldots, a_{d}, a_{d+1}\Big) & \,\in\, 
\Big(\big(\mathsf{R}\mathsf{T}_1\big)_{d+1}(d)\Big)_{\mathsf{Aff}}  
\, \subset\, \mathcal{F}_d^* \times \mathbb{C}^2 \times \mathbb{C}^{d+1}. 
\end{align*}
To prove \cref{k_clsr2_ag2}, it suffices to show the following: given small a neighbourhood $U$ of
$\left(f,\, (0,\,0),\, a_1,\, \cdots,\, a_{d},\, a_{d+1}\right)$, there exists an element
\begin{align*}
\Big(F,\, (s,\,t),\, A_1,\, \cdots,\, A_{d},\, A_{d+1}\Big) \,\in\, 
\Big(\mathsf{A}_1^{\mathsf{F}} \underbrace{\mathsf{T}_{0} \cdots \mathsf{T}_{0}}_{d-\textnormal{times}} \mathsf{F}\Big)_{\mathsf{Aff}} \cap U. 
\end{align*}

We will follow the previous idea.
Note that $F$ passes through the points $(A_1,\, 0),\, \cdots,\, (A_{d},\, 0)$. 
Hence, $F$ has to be of the form
\begin{align*}
F(x,y)&\ =\ (x-A_1) \cdots (x-A_{d-1}) (x-A_d) C_0 + y \Phi(x,\,y),
\end{align*}
where $\Phi$ is a polynomial of degree $d-1$. Write $\Phi$ explicitly as 
\begin{align*}
\Phi(x,\,y) &\,:=\, B_{00} + B_{10} x + B_{01} y + \cdots + \frac{B_{0, d-1}}{(d-1)!} y^{d-1}.
\end{align*}
Since $U$ is a sufficiently small open neighbourhood of 
$\left( f,\, \left(0,\,0\right), \,a_1, \,\cdots,\, a_{d-1},\, a_d,\, a_{d+1}\right)$,
it follows that $C_0$, $B_{00}$ and $B_{10}$ are small while
$B_{20}$ and $B_{01}$ are non-zero. 
Furthermore, $A_1,\, \cdots,\, A_{d}$ 
are all distinct and different from $0$. 
We now wish to find a point $(s,\,t)$ close to the origin, such that $t\,\neq\, 0$ (i.e., the point
$(s,\,t)$ does not lie on the $\mathtt{x}$--axis) 
and is a nodal point of the curve defined by the zero set of $F$. Hence, we need to solve the following
set of equations 
\begin{align}
F(s,\,t) & \,=\, 0, \qquad F_x(s,\,t) \,=\, 0 \qquad \textnormal{and} \qquad F_y(s,\,t) \,=\, 0, \nonumber \\   
\qquad \textnormal{where} \qquad t & \,\neq\, 0 \qquad \textnormal{and} \qquad (s,\,t) \quad \textnormal{is small}.   
\label{ch_step3_node_eqn_ag2}
\end{align}   
To make the computations more convenient, rewrite $F(x,\,y)$ as follows: 
\begin{align} 
F(x,\,y)& \,=\, Q(x) C_0 + y (B_{00} + B_{10} x + P(x,\,y)),  \nonumber \\ 
\qquad \textnormal{where} \quad & Q(x)\,:=\, (x-A_1) \cdots (x-A_{d}), \nonumber \\  
 \qquad \textnormal{and }& P(x,\,y)\,:=\, B_{01} y + \frac{B_{20}}{2} x^2 + B_{11} xy +
\frac{B_{02}}{2} y^2 + \cdots + \frac{B_{0, d-1}}{(d-1)!} y^{d-1}. \label{F_rewrite_ag2}
\end{align}
Using \cref{ch_step3_node_eqn_ag2,F_rewrite_ag2}, 
solve for $C_0$, $B_{00}$ and $B_{10}$ to conclude that
\begin{align}
C_0 & \,=\, \frac{t^2 P_y(s,t)}{Q(s)},\nonumber \\ 
B_{00} & \,=\, -\Big(\frac{P(s,t)Q(s)+t Q(s) P_y(s,t) -st Q^{\prime}(s)P_y(s,t)-sQ(s)P_x(s,t)}{Q(s)}\Big), \quad \textnormal{and} \nonumber \\ 
B_{10}& \,=\, -\Big(\frac{tQ^{\prime}(s)P_y(s,t) + Q(s) P_x(s,t)}{Q(s)}\Big). \label{sol_ch_step3_ag985}
\end{align}
This proves \cref{k_clsr2_ag2}. 

Next, to compute the multiplicity of the intersection, first note that \textit{every} curve close to 
$f$ is given by \cref{sol_ch_step3_ag985}. To compute the multiplicity of the intersection, we have to    
compute $F(a_{d+1}, \,0)$, where $a_{d+1}$ is distinct from 
$0,\, a_1,\, \cdots,\, a_{d}$. Using \cref{F_rewrite_ag2,sol_ch_step3_ag985} it follows that 
\begin{align}
F(a_{d+1}, 0) & \,=\, Q(a_{d+1}) \frac{B_{01}}{Q(0)} t^2 + O(t^3). \label{mult_CH_step3_ag2}
\end{align} 
Since 
\[Q(a_{d+1}) \,\neq\, 0, \qquad Q(0) \,\neq\,0 \qquad  \textnormal{and} \qquad  
B_{01} \,\neq\, 0,\] 
it follows from \cref{mult_CH_step3_ag2} that the 
multiplicity of the intersection is $2$. 
The multiplicity computation completes the
justification for 
the fourth term on the right-hand side  
of \eqref{many_Tks_A1F_pp_Break_Fin} (on the level of cycles).

It remains to justify the third term on the right-hand side of \cref{many_Tks_A1F_pp_Break_Fin}. 
We need to show that 
\begin{align}
\big(\mathsf{R}\mathsf{A}_1^{\mathsf{F}}\big)_{d+1}(d) & 
\ \subset\ 
\overline{\mathsf{A}_1^{\mathsf{F}} \underbrace{\mathsf{T}_{0} \cdots \mathsf{T}_{0}}_{d-\textnormal{times}} \mathsf{F}}. \label{k_clsr_ag5}
\end{align}
Let $(H_1,\, H_d,\, p,\, x_1,\, \cdots,\, x_{d},\, x_{d+1})$ be an arbitrary point of 
$\big(\mathsf{R}\mathsf{A}_1^{\mathsf{F}}\big)_{d+1}(d)$. 
Similarly, given a small enough neighbourhood of the point, we need to find an element of it such
that it belongs to the open part of the rightmost space of \cref{k_clsr_ag5} as well. We now descend to
the affine space. Following \Cref{convention}, assume that $(a,b)$ is a nodal point. Since the nodal point
does not lie on the line, it follows that $b \,\neq\, 0$.  Since $H_d$ is a reducible curve, with the line $H_1$
being one of the reducible components, the curve $f$ is given by 
\begin{align*}
f(x,\,y)& \ =\ y \varphi(x,\,y),
\end{align*} 
where $\varphi(x,\,y)$ is a polynomial of degree $d-1$. Write down $\varphi(x,y)$ explicitly as 
\begin{align*}
\varphi(x,\,y)&\ :=\ b_{00} + b_{10} x + b_{01} y + \cdots +\frac{b_{0, d-1}}{(d-1)!} y^{d-1}. 
\end{align*}
Note that $(a,\,b)$ is a nodal point of $\varphi$. Hence, 
\begin{align}
\varphi(a,\,b)&\,=\, 0, \qquad \varphi_x(a,\,b) \,=\, 0 \qquad \textnormal{and} \qquad \varphi_y(a,\,b)
\,=\,0. \label{varphi_node_ab} 
\end{align}

As before, given an open neighbourhood $U$ of $\left(f,\, (a,\, b),\, a_1,\, \cdots,\, a_{d},\,
a_{d+1}\right)$, we will find an element 
\begin{align*}
\Big(F,\, (A,\, B),\, A_1,\, \cdots,\, A_{d},\, A_{d+1}\Big) \,\in\, 
\Big(\mathsf{A}_1^{\mathsf{F}} \underbrace{\mathsf{T}_{0} \cdots \mathsf{T}_{0}}_{d-\textnormal{times}}
\mathsf{F}\Big)_{\mathsf{Aff}} \cap U. 
\end{align*}
The curve $F$ passes through the points $(A_1,\, 0),\, \cdots,\, (A_{d}, \,0)$. 
Hence, $F$ has to be of the form
\begin{align*}
F(x,y)& \ =\ (x-A_1) \cdots (x-A_{d-1}) (x-A_d) C_0 + y \Phi(x,\,y),
\end{align*}
where $\Phi$ is a polynomial of degree $d-1$. Write down $\Phi$ explicitly as 
\begin{align*} 
\Phi(x,\,y) &\ :=\ B_{00} + B_{10} x + B_{01} y + \cdots + \frac{B_{0, d-1}}{(d-1)!} x^{d-1}.
\end{align*}
Since $U$ is a sufficiently small open neighbourhood of $\left(f,\, (a,\, b),\, a_1,\, \cdots,\,
a_{d},\, a_{d+1}\right)$, it follows that $C_0$ $B_{00}$, $B_{10}$ and $B_{01}$ are small. 
Furthermore, $A_1,\, \cdots,\, A_{d}$ are all distinct. 
We now wish to find a point $(A,\, B)$ close to $(a,\,b)$, such that it 
is a nodal point of the curve defined by the zero set of $F$.
Hence, we need to solve the following set of equations 
\begin{align}
F(A,\, B)  \,=\, 0, \qquad F_x(A,\, B) & \,=\, 0, \qquad \textnormal{and} \qquad F_y(A,\, B) \,=
\,0, \nonumber \\   
& \textnormal{where $(A,\,B)$ is close to $(a,\,b)$}.   
\label{ch_step3_node_eqn_ag5}
\end{align}   
Rewrite $F(x,\,y)$ as follows: 
\begin{align} 
F(x,\,y)& \,=\, Q(x) C_0 + y (B_{00} + B_{10} x + B_{01}y + \alpha(x) x^2 + \beta(y)y^2 + \gamma(x,\,
y) x y ),  \nonumber \\ 
\textnormal{where \, }& Q(x)\,:=\, (x-A_1) \cdots (x-A_{d}),  \nonumber \\  
 & \alpha(x) \,:=\, \frac{B_{20}}{2} + \frac{B_{30} x}{6} + \cdots ,  \nonumber \\ 
& \beta(y)\,:=\, \frac{B_{02}}{2} + \frac{B_{03} y}{6} + \cdots, \nonumber \\ 
& \gamma(x,\,y)\,:=\, B_{11} + \frac{B_{21}}{2} x + \frac{B_{12}}{2} y + \cdots.  \label{F_rewrite_ag5}
\end{align}
Using \cref{ch_step3_node_eqn_ag5,F_rewrite_ag5}, 
solve for $B_{00}$, $B_{10}$ and $B_{01}$ to conclude that 
\begin{align}
B_{00}\, =\,& A^2 \alpha(A) + B^2 \beta(B) + AB \gamma(A,\,B)-\frac{2C_0 Q(A)}{B} \nonumber \\ 
       & +A^3 \alpha^{\prime}(A) + B^3 \beta^{\prime}(B) + \frac{A C_0 Q^{\prime}(A)}{B}   \nonumber \\ 
       & + AB^2 \gamma_{y}(A,B) + A^2 B \gamma_{x}(A,\,B), \nonumber \\  
B_{10} \,=\,&  -2A\alpha(A) -B\gamma(A,\,B) -A^2 \alpha^{\prime}(A)-\frac{C_0 Q^{\prime}(A)}{B}
-AB\gamma_{x}(A,\,B), \quad \textnormal{and} \nonumber \\ 
B_{01}\,=\,&  -2B\beta(B) -A\gamma(A,B)+\frac{C_0 Q(A)}{B^2}-B^2 \beta^{\prime}(B)-
AB\gamma_y(A,\,B). \label{sol_ch_step3_ag5j}
\end{align}
Note that $B\,\neq\, 0$ was needed to solve the equations. Substituting $C_0\,=\,0$ in the expression 
for $B_{00}, B_{10}$ and $B_{01}$, we precisely get the expression for $b_{00}, b_{10}$ and $b_{01}$, with 
$b_{ij}$ replaced by $B_{ij}$. This follows from the fact that $(a,\,b)$ is a nodal point of $\varphi$ 
(i.e., by using \cref{varphi_node_ab}). Hence, the $B_{00}, B_{10}$ and $B_{01}$ 
we have constructed are indeed close to $b_{00}, b_{10}$ and $b_{01}$.    
This proves \cref{k_clsr_ag5}. 

Next, let us compute the multiplicity of the intersection. Note that \textit{every} curve close to 
$f$ is given by \cref{sol_ch_step3_ag5j}. For intersection multiplicity here, we need to    
compute $F(a_{d+1}, \,0)$, where $a_{d+1}$ is distinct from 
$a_1,\, \cdots,\, a_{d}$. Using \cref{F_rewrite_ag5}, it follows that 
\begin{align}
F(a_{d+1}, 0) & \,=\, Q(a_{d+1}) C_0 \nonumber \\ 
              & =\, Q(a_{d+1}) C_0  + O(C_0^2). \label{mult_CH_step3_ag5}
\end{align} 
Since $Q(a_{d+1}) \,\neq\, 0$, 
the required intersection multiplicity is $1$ as follows from \cref{mult_CH_step3_ag5}.

The third term on the right-hand side of \cref{many_Tks_A1F_pp_Break_Fin} is now justified. This completes
the proof of \Cref{theorem_for_many_Tks_A1F_pp_Break_Fin}. 
\end{proof}

\subsection{Numerical invariants: enumerating curves with a free node}\label{enumerative applications}

Using the theorems proved above \cref{CH_degen_arg,na1_sm,na1_d_pts} will be derived.

Using \cref{diag_inc_std}, \Cref{step1_CH_p} can be rewritten as
\begin{align} 
\big[\mathsf{A}_1^{\mathsf{F}} \mathsf{F}\big]\cdot(y_1 + a_1)\cdot (y_d + d a_d) & \,=
\,\, \big[\mathsf{A}_1^{\mathsf{F}}\mathsf{T}_0\big]. 
\label{CH_Step1_formula}
\end{align}
Given a positive integer $t$, define the following cycle 
\begin{align}\label{mu_t}
\mu_t\,\,:= \,\,y_1^2 y_d^{\delta_d-(t+1)} a_1.
\end{align}
Now intersect both sides of \cref{CH_Step1_formula} with $\mu_t$ with $t=1$. Using the fact that $y_1^3 \,=\,0$ and $a_1^3\,=\,0$, the equality simplifies to
\begin{align} 
\big[\mathsf{A}_1^{\mathsf{F}} \mathsf{F}\big]\cdot y_1^2 y_d^{\delta_d-1} a_1^2 & \,=\,\, 
\big[\mathsf{A}_1^{\mathsf{F}}\mathsf{T}_0\big] \cdot y_1^2 y_d^{\delta_d-2} a_1. \label{CH_Step1_formula_num_simp}
\end{align}
The left-hand side of \cref{CH_Step1_formula_num_simp} is an intersection
number inside $\mathsf{M}^1_1$. We are counting a line, a degree $d$ curve and two points $p$ and $x_1$, such that the point $p$ 
is a nodal point of the curve and the point $x_1$ is free (it need not lie on the line or the cubic). Intersecting with 
$y_1^2$ fixes the line. Intersecting with $a_1^2$ fixes the free point. Intersecting with $y_d^{\delta_d-1}$
gives a nodal curve of degree $d$. Hence, the left-hand side of
\cref{CH_Step1_formula_num_simp} is precisely equal to $\mathsf{N}_d(\mathsf{A}_1)$.

Consider the right-hand side of \cref{CH_Step1_formula_num_simp}. Intersecting with $y_1^2$ 
fixes the line. However, now we are intersecting with $a_1$. This fixes the location of the point $x_1$ to be 
on the line. Finally, intersecting with $y_d^{\delta_d-2}$ gives the number of nodal degree $d$ curves 
passing through $\delta_d-2$ generic points. Hence, the number on the right-hand side is precisely equal to 
$\mathsf{N}_d(\mathsf{A}_1, 1)$. Consequently, \cref{CH_Step1_formula_num_simp} is exactly the same as the 
first equality of \cref{CH_degen_arg}.

As before, \Cref{theorem_for_many_Tks_A1F_pp} can be rewritten as follows
\begin{align}\label{CH_Step_n_formula}
\big[\mathsf{A}_1^{\mathsf{F}} \underbrace{\mathsf{T}_{0} \ldots \mathsf{T}_{0}}_{n-\textnormal{times}}\mathsf{F}\big]\cdot 
\big(y_1 + a_{n+1}\big)\cdot \big(y_d + d a_{n+1}\big) \,\,&= \,\,
\big[\mathsf{A}_1^{\mathsf{F}} \underbrace{\mathsf{T}_{0}\ldots \mathsf{T}_{0}}_{(n+1)-\textnormal{times}}\big]\nonumber \\ 
& + \sum_{i=1}^{n} \big[\mathsf{A}_1^{\mathsf{F}}
\underbrace{\mathsf{T}_{0}\ldots \mathsf{T}_{0}}_{n-\textnormal{times}} \mathsf{F}\big] \cdot \big(a_i^2 + a_i a_{n+1} + a_{n+1}^2\big).
\end{align}
Now intersect both sides of \cref{CH_Step_n_formula} with $\mu_{n+1}$ (cf. \cref{mu_t}) and simplify. This yields
\begin{align} 
\big[\mathsf{A}_1^{\mathsf{F}} \underbrace{\mathsf{T}_{0} \ldots \mathsf{T}_{0}}_{n-\textnormal{times}}\mathsf{F}\big] 
\cdot y_1^2 y_d^{\delta_d-(n+1)}\cdot (a_1 \ldots a_n) \cdot a_{n+1}^2 & = 
\big[\mathsf{A}_1^{\mathsf{F}} \underbrace{\mathsf{T}_{0} \ldots \mathsf{T}_{0}}_{(n+1)-\textnormal{times}}\big]
\cdot y_1^2 y_d^{\delta_d-(n+2)}\cdot (a_1 \ldots a_{n+1}). 
\label{CH_Step_n_formula_num_simp}
\end{align}
Consider the left-hand side of \cref{CH_Step_n_formula_num_simp}. 
Recall that this is an intersection 
number inside $\mathsf{M}^1_{n+1}$. We are counting a line, a degree $d$ curve and $(n+2)$ points, namely
$p$, $x_1,\, x_2,\, \cdots,\, x_{n+1}$, such that the point $p$ is a nodal point of the curve, the points
$x_1$ to $x_n$ are transverse points of intersection of the line and the curve 
while the point $x_{n+1}$ is free.

Note that intersecting with 
$y_1^2$ fixes the line, and intersecting with $a_{n+1}^2$ fixes the free point. 
Intersecting with $y_d^{\delta_d-(n+1)}$ makes the nodal curve pass through
$\delta_d-(n+1)$ generic point. Finally, intersecting with $(a_1 \ldots a_n)$ fixes the location of the points 
$x_1, \cdots x,\,_n$ on the line. Hence, the left-hand side of \cref{CH_Step_n_formula_num_simp}
is precisely equal to $\mathsf{N}_d(\mathsf{A}_1, n)$,
the number of nodal degree $d$ curves passing through $\delta_d-(n+1)$
generic points and $n$ generic points on the line. Similarly, the right-hand side of \cref{CH_Step_n_formula_num_simp} 
is equal to $\mathsf{N}_d(\mathsf{A}_1, n)$. Hence,
\cref{CH_Step_n_formula_num_simp} implies --- as we vary $n$ from $0$ to $d-2$ --- all
the equalities of \cref{CH_degen_arg}.

\Cref{na1_sm} can be obtained using \Cref{theorem_for_many_Tks_A1F_pp_Break}. 
Intersect both sides of \cref{many_Tks_A1F_pp_Break} with $\mu_d$ (cf. \cref{mu_t}). 
After simplifying, this gives us 
\begin{align} 
\big[\mathsf{A}_1^{\mathsf{F}} \underbrace{\mathsf{T}_{0} \ldots
\mathsf{T}_{0}}_{(d-1)-\textnormal{times}}\mathsf{F}\big] 
\cdot y_1^2 y_d^{\delta_d-d}\cdot
(a_1 \ldots a_{d-1}) \cdot a_{d}^2 & \,=\, 
\big[\mathsf{A}_1^{\mathsf{F}} \underbrace{\mathsf{T}_{0} \ldots \mathsf{T}_{0}}_{d-\textnormal{times}}\big]
\cdot y_1^2 y_d^{\delta_d-(d+1)}\cdot
(a_1 \ldots a_{d}) \nonumber \\ 
& + \big[\mathsf{R}_d(d)\big]\cdot y_1^2 y_d^{\delta_d-(d+1)} a_1 a_2 \ldots a_d. 
\label{CH_Step_n_formula_num_simp_agg}
\end{align}
As before, the left-hand side of \cref{CH_Step_n_formula_num_simp_agg} 
is $\mathsf{N}_d(\mathsf{A}_1, d-1)$. Similarly, the first 
term on the right-hand side of 
 \cref{CH_Step_n_formula_num_simp_agg} is 
$\mathsf{N}_d(\mathsf{A}_1, d)$. Now analyze the second term on the right-hand side, namely 
\[ \big[\mathsf{R}_d(d)\big]\cdot y_1^2 y_d^{\delta_d-(d+1)} a_1 a_2 \ldots a_d.\]
As before, intersecting with $y_1^2$ fixes the line. Note that 
\[\delta_{d}-(d+1) = \delta_{d-1}. \]
Hence, intersecting with $y_d^{\delta_{d-1}}$ fixes the degree $(d-1)$ component of the curve 
(the line component has already been fixed, since we have intersected with $y_1^2$). 
However, corresponding to each such degree $(d-1)$ curve, there are $ (d-1)$ choices for the nodal point $p$ 
(it is one of the points of intersection of the curve with the line). Hence, this number is precisely equal to 
\[ (d-1)\mathsf{N}_{d-1}(\mathsf{S}). \]
This gives us the degeneration formula of \cref{na1_sm}.

\Cref{na1_d_pts} is obtained by intersecting both sides of \cref{many_Tks_A1F_pp_Break_Fin} with $\mu_{d+1}$ (cf. \cref{mu_t}).
After simplifying, this yields the following:
\begin{align} 
\big[\mathsf{A}_1^{\mathsf{F}} \underbrace{\mathsf{T}_{0} \ldots \mathsf{T}_{0}}_{d-\textnormal{times}}\mathsf{F}\big] 
\cdot y_1^2 y_d^{\delta_d-(d+1)}\cdot
(a_1 \ldots a_{d}) \cdot a_{d+1}^2 & = 
\sum_{i=1}^d \big[\mathsf{R}^i_{d+1}(d)\big] \cdot \mu_4 \nonumber \\ 
& + \big[\big(\mathsf{R}\mathsf{A}_1^{\mathsf{F}}\big)_{d+1}(d)\big] \cdot \mu_4 + 
2 \big[\big(\mathsf{R}\mathsf{T}_1\big)_{d+1}(d)\big] \cdot \mu_4.\label{CH_Step_n_formula_num_simp_agg_pk}
\end{align}
As before, the left-hand side of \cref{CH_Step_n_formula_num_simp_agg_pk} 
is simply $\mathsf{N}_d(\mathsf{A}_1, d)$. Now analyze the right-hand side of
\cref{CH_Step_n_formula_num_simp_agg_pk}. After simplification it is obtained that the first term on the 
right-hand side is 
\begin{align*}
\sum_{i=1}^d \big[\mathsf{R}^i_{d+1}(d)\big] \cdot \mu_4 & \,=\, d \mathsf{N}_{d-1}(\mathsf{S}). 
\end{align*} 
Next, after simplification, it is obtained that the second term on the right-hand side is equal to 
\begin{align*} 
\big[\big(\mathsf{R}\mathsf{A}_1^{\mathsf{F}}\big)_{d+1}(d)\big] \cdot \mu_4 & \,=\, \mathsf{N}_{d-1}(\mathsf{A}_1). 
\end{align*}
Finally, note that 
\begin{align*} 
\big[\big(\mathsf{R}\mathsf{T}_1\big)_{d+1}(d)\big] \cdot \mu_4 & = \mathsf{N}_{d-1}(\mathsf{T}_1). 
\end{align*}
Combining these facts, we obtain the degeneration formula 
given by \cref{na1_d_pts}. \\

\subsection{Numerical invariants: enumerating curves with the node lying on a line}

Now we will prove \cref{mr_11} (first part of \Cref{mr1}). It is enough to establish the equalities of 
\cref{CH_degen_arg_L,na1_sm_L,na1_d_pts_L}.

Given an integer $t\geq 0$, define the following cycle:
\begin{align*}
\mu_t^{\mathsf{L}}\ :=\ y_1^2 y_d^{\delta_d-(t+2)} b_1 a_1 a_2 \cdots a_n a_{n+1}.
\end{align*}
Intersecting both sides of \cref{CH_Step_n_formula} with $\mu_{n+1}^{\mathsf{L}}$, and a similar analysis
as above, produce \cref{CH_degen_arg_L}.\\
\hf In order to get \cref{na1_sm_L}, 
intersect both sides of \cref{many_Tks_A1F_pp_Break} with $\mu_d^{\mathsf{L}}$; that will give 
\cref{na1_sm_L}. 

Similarly, to get \cref{na1_d_pts_L}, intersect both sides of \cref{many_Tks_A1F_pp_Break_Fin} with
$\mu_{d+1}^{\mathsf{L}}$. This produces the required equality.

\section{Extension of Caporaso-Harris to one cuspidal curves} 
\label{CH_cusp_general}

We will now extend the degeneration arguments presented in \Cref{CH_one_node_review} 
and solve the following problem: 
\begin{center}
{\it How many degree $d$ curves are there in $\mathbb{P}^2$ that have a cusp and pass through $\delta_d-2$ 
generic points?}
\end{center}
Let $\mathsf{N}_d(\mathsf{A}_2, k)$ denote the number of cuspidal degree $d$ 
curves, passing through $\delta_d-2-k$ generic points and 
$k$ points on a line. When $k\,=\,0$, we abbreviate and omit writing the $k$, i.e., 
$\mathsf{N}_d(\mathsf{A}_2)\,:=\, \mathsf{N}_d(\mathsf{A}_2,0)$. 
We will show that 
\begin{align}
\mathsf{N}_d(\mathsf{A}_2) & \,=\, \mathsf{N}_d(\mathsf{A}_2,\, 1), \nonumber \\ 
\mathsf{N}_d(\mathsf{A}_2, 1) & \,=\, \mathsf{N}_d(\mathsf{A}_2,\, 2), \nonumber \\ 
\ldots & \nonumber \\ 
\mathsf{N}_d(\mathsf{A}_2, d-2) & \,=\, \mathsf{N}_d(\mathsf{A}_2,\, d-1). \label{CH_degen_arg_cusp}
\end{align}
In other words, 
\begin{align} 
\mathsf{N}_d(\mathsf{A}_2) & \,=\, \mathsf{N}_d(\mathsf{A}_2,\, d-1). \label{na1_d_minus_1_cusp}
\end{align}
Hence, the number of cuspidal degree $d$ curves passing through $\delta_d-2$ generic points is equal to 
the number of cuspidal degree $d$ curves passing through $\delta_d-d-1$ generic points and $(d-1)$ 
collinear points. We further show that
\begin{align}
\mathsf{N}_d(\mathsf{A}_2, d-1) & \,=\, \mathsf{N}_d(\mathsf{A}_2, d) + 3\mathsf{N}_{d-1}(\mathsf{T}_1); \label{na2_sm}
\end{align}
\begin{align}
\mathsf{N}_d(\mathsf{A}_2, d) & \,=\, \mathsf{N}_{d-1}(\mathsf{A}_2) + 
3\mathsf{N}_{d-1}(\mathsf{A}_1^{\mathsf{L}}) + 3d \mathsf{N}_{d-1}(\mathsf{T}_1^{\mathsf{Pt}})
+2 \mathsf{N}_{d-1}(\mathsf{T}_2). 
\label{na2_d_pts}
\end{align}
Equations \eqref{na1_d_minus_1_cusp}, \eqref{na2_sm} and \eqref{na2_d_pts} together imply \cref{mr_12}. 
Let us now prove these assertions. 
As before, certain cycle level statements are proved first; then we intersect them with cycles of
complementary dimensions to get equality of numbers. The cycle level statements are analogs
of \Cref{step1_CH_p}, \Cref{theorem_for_many_Tks_A1F_pp},
\Cref{theorem_for_many_Tks_A1F_pp_Break} and \Cref{theorem_for_many_Tks_A1F_pp_Break_Fin}. Here we consider
cuspidal curves instead of the nodal ones.

In \Cref{CH_one_node_review}, a few subspaces of $\mathsf{M}_{n}^1$ (cf. \cref{definition_A_1T_0^n,definition_A_1T_0^nF,definition_R_n,definition_RT_1,definition_RA_1,definition_R^i_n}) are considered. In all the definitions there $\mathsf{A}_1$ stands for curves having only one node. Here, we consider curves having one cusp only (that is, $\mathsf{A}_2$ singularity). We can consider similar subspaces for cuspidal curves, and replace $A_1$ by $A_2$ in the notation. For example,
\[ \mathsf{A}_2^{\mathsf{F}} \underbrace{\mathsf{T}_{0}\cdots \mathsf{T}_{0}}_{n-\textnormal{times}}\,\,
\subset\, \,\mathsf{M}_{n}^1.\] 
The definitions remain the same, but only the nodal condition will be replaced by the cuspidal one.
 
The following result is analogous to \Cref{step1_CH_p}.
 
\begin{thm} 
\label{step1_CH_p_cusp}
The following equality of homology classes hold in $\mathsf{M}^1_1(d)$:
\begin{align*}
\big[\mathsf{A}_2^{\mathsf{F}} \mathsf{F}\big]\cdot
\big[\big(x_1 \in \mathsf{L}\big)\big] \cdot \big[ \big(x_1 \in \mathcal{C}_d\big) \big] 
&\,= \,\big[\mathsf{A}_2^{\mathsf{F}}\mathsf{T}_0\big], 
\end{align*}
where $d \,\in\, \mathbb{Z}^{\geq 3}$. 
\end{thm}

\Cref{step1_CH_p_cusp} is a special case of the following analog of \Cref{theorem_for_many_Tks_A1F_pp}.

\begin{thm}
\label{theorem_for_many_Tks_A1F_pp_cusp}
Take $d \,\in\, \mathbb{Z}^{\geq 3}$, and let $n$ be an integer such that 
\[0 \,\leq\, n \,\leq\, d-2.\]
Then, on $\mathsf{M}^1_{n+1}(d)$ 
the following equality of homology classes hold:
\begin{align}\label{many_Tks_A1F_pp_cusp}
\big[\mathsf{A}_2^{\mathsf{F}} \underbrace{\mathsf{T}_{0} \cdots
\mathsf{T}_{0}}_{n-\textnormal{times}}\mathsf{F}\big]\cdot 
\big[\big(x_{n+1}\in \mathsf{L}\big)\big]\cdot \big[\big(x_{n+1}\in \mathcal{C}_d\big)\big]&\,\, = \,\,
\big[\mathsf{A}_2^{\mathsf{F}} \underbrace{\mathsf{T}_{0}\cdots \mathsf{T}_{0}}_{(n+1)-\textnormal{times}}\big]\nonumber \\ 
& + \sum_{i=1}^{n} \big[\mathsf{A}_2^{\mathsf{F}}
\underbrace{\mathsf{T}_{0}\cdots \mathsf{T}_{0}}_{n-\textnormal{times}} \mathsf{F}\big] \cdot \big[\big(x_i = x_{n+1}\big)\big].
\end{align}
\end{thm}
\begin{rem}
\Cref{CH_degen_arg_cusp} can be obtained from \cref{many_Tks_A1F_pp_cusp} by intersecting both sides with 
\[\mu_{n+1}:= y_1^2 y_d^{\delta_d-(n+3)} a_1 a_2 \cdots a_n a_{n+1}.\] 
The reasoning is identical to how \cref{many_Tks_A1F_pp} implies \cref{CH_degen_arg} described in \Cref{enumerative applications}.
\end{rem}

\noindent\begin{proof}[{\textbf{Proof of Theorem \ref{theorem_for_many_Tks_A1F_pp_cusp}}}]
The proof is almost the same as 
the proof of \Cref{theorem_for_many_Tks_A1F_pp}. 
First, we need to show that if $d\,\geq\, 2$, then $\mathsf{A}_2^{\mathsf{F}}$
is a smooth manifold of codimension $4$. This has been proved in
\cite[p.~22, Proposition 6.1]{Tang_IBRMAPAC}. 

The rest of the proof is the same if $n \,\leq\, d-3$. When $n\,=\,d-2$, there is a small
point we need to observe. 
By a dimension counting argument, we can't rule out a reducible curve, one of whose component is the line. 
However, when we intersect with generic constraints, we will require the degree $d-1$ curve to intersect the 
line transversally. This is not possible for a simple geometric reason; a cuspidal point of the degree $d$ 
curve can't become a nodal point (it can become a tacnode, but not a node). Hence, the reducible configuration 
is not possible, even when $n\,=\,d-2$. 
\end{proof}

\begin{thm}
\label{theorem_for_many_Tks_A1F_pp_Break_cusp}
Let $d \,\in \,\mathbb{Z}^{\geq 3}$. The following equality of homology classes 
holds in $\mathsf{M}_{{d}}^1(d)$:
\begin{align}\label{many_Tks_A1F_pp_Break_cusp}
\big[\mathsf{A}_2^{\mathsf{F}} \underbrace{\mathsf{T}_{0} \cdots
\mathsf{T}_{0}}_{(d-1)-\textnormal{times}} \mathsf{F}\big]\cdot 
\big[\big(x_d \in \mathsf{L}\big)\big]\cdot \big[\big(x_d \in \mathcal{C}_d\big)\big]&\,\, = \,\,
\big[\mathsf{A}_2^{\mathsf{F}} \underbrace{\mathsf{T}_{0}\cdots \mathsf{T}_{0}}_{d-\textnormal{times}}\big] \nonumber \\ 
&+ \sum_{i=1}^{d-1}\big[\mathsf{A}_2^{\mathsf{F}}
\underbrace{\mathsf{T}_{0}\cdots \mathsf{T}_{0}}_{(d-1)-\textnormal{times}}\mathsf{F}\big] \cdot \big[\big(x_i = x_d\big)\big] 
\nonumber \\
& + 3 \big[\big(\mathsf{R}\mathsf{T}_1\big)_d(d)\big].
\end{align}
\end{thm}

\begin{rem}
As before, intersecting both sides of 
\cref{many_Tks_A1F_pp_Break_cusp} with 
\[\mu_{d}\ :=\ y_1^2 y_d^{\delta_d-(d+2)} a_1 a_2 \cdots a_d,\] 
\cref{na2_sm} can be obtained.
\end{rem}

\begin{proof}[{\textbf{Proof of \Cref{theorem_for_many_Tks_A1F_pp_Break_cusp}}}] 
The first and second terms on the right-hand side of \cref{many_Tks_A1F_pp_Break_cusp} 
are justified as before. The third term needs to be justified.
We continue with the set-up of the proof of Theorem 
\ref{theorem_for_many_Tks_A1F_pp_Break}. Following the same thread of reasoning, it suffices to show the
following:
\begin{align}
\big(\mathsf{R}\mathsf{T}_1\big)_d(d) & 
\ \subset\
\overline{\mathsf{A}_2^{\mathsf{F}} \underbrace{\mathsf{T}_{0} \ldots
\mathsf{T}_{0}}_{(d-1)-\textnormal{times}} \mathsf{F}}. \label{k_clsr77}
\end{align}

Let $\left(H_1,\, H_d,\, p,\, x_1,\, \cdots,\, x_{d-1},\, x_d\right)$ be an arbitrary point of 
$\big(\mathsf{R}\mathsf{T}_1\big)_d(d)$. It will be shown that every open neighbourhood of 
$(H_1, H_d, x_1, \ldots, x_{d-1}, x_d)$
intersects 
\[\mathsf{A}_2^{\mathsf{F}} \underbrace{\mathsf{T}_{0} \cdots \mathsf{T}_{0}}_{(d-1)-\textnormal{times}} \mathsf{F}.\]
Switching to an affine space, assume $H_1$ to be the $\mathtt{x}$--axis, and the point of tangency $p$ is
the origin. Furthermore, assume that $x_1,\, x_2,\, \cdots,\, x_{d-1}$ are given by
$(a_1,\, 0)$, $(a_2,\, 0),\, \cdots,\, (a_{d-1},\,0)$. Since these points are all distinct and also not
equal to $p$, it follows that $a_1,\, a_2,\, \cdots,\, a_{d-1}$ and $0$ are all distinct.

Note that $H_1$ being a component, the reducible curve $H_d$ is given by a polynomial of the form
\begin{align*}
f(x,\,y)&\ =\ y \varphi(x,\,y),
\end{align*} 
where $\varphi(x,\,y)$ is a polynomial of degree $d-1$. 
Let us write down $\varphi(x,\,y)$ explicitly as 
\begin{align}
\varphi(x,y)\,\, &:=\,\, b_{00} + b_{10} x + b_{01} y + \cdots +\frac{b_{0, d-1}}{(d-1)!} y^{d-1}. \nonumber
\end{align}
Recall that the degree $d-1$ component is tangent of first order to the line at the origin, and the origin is a smooth
point of the degree $d-1$ component. Hence, $b_{00} \,=\,0$, $b_{10}\,=\,0$; but $b_{20}
\,\neq\, 0$ and $b_{01} \,\neq\, 0$.
Note that 
\begin{align} 
\Big(f,\, (0,\,0),\, a_1,\, \cdots,\, a_{d-1},\, a_d\Big) & \,\in\, 
\Big(\big(\mathsf{R}\mathsf{T}_1\big)_d(d)\Big)_{\mathsf{Aff}}
\,\subset\, \mathcal{F}_d^* \times \mathbb{C}^2 \times \mathbb{C}^d. \nonumber
\end{align}
Given any sufficiently open neighbourhood $U$ of $\left(f,\, (0,\,0),\, a_1,\, \cdots,\,
a_{d-1},\, a_d\right)$, we have to show the existence of an element
\begin{align*}
\Big(F,\, (s,\,t),\, A_2,\, \cdots, \,A_{d-1},\, A_d\Big) \,\in\, 
\Big(\mathsf{A}_2^{\mathsf{F}} \underbrace{\mathsf{T}_{0} \cdots \mathsf{T}_{0}}_{(d-1)-\textnormal{times}} \mathsf{F}\Big)_{\mathsf{Aff}} \cap U. 
\end{align*}

The curve $F$, if it exists, 
passes through the points $(A_1,\, 0),\, \cdots,\, (A_{d-1},\, 0)$. Hence, it
has to be of the form 
\begin{align*}
F(x,\,y)\,& =\, Q(x)(C_0 + C_1 x)+ y(B_{00} + B_{10} x + x^2 \alpha(x) + y \beta(x,\,y)), \qquad \textnormal{where} \\ 
Q(x)\,&:=\, (x-A_1)(x-A_2) \cdots (x-A_{d-1}), \nonumber \\ 
\alpha(x)\, &:=\, \frac{B_{20}}{2} x^2 + \frac{B_{30}}{6} x^3 + \cdots \qquad \textnormal{and} 
\qquad \beta(x,\,y)\,:=\, B_{01} + B_{11} x + \cdots. 
\end{align*}
We note that $C_{0},\, C_1,\, B_{00}$ and $B_{10}$ are small. 
Furthermore, $\alpha(0) \,=\, \frac{B_{20}}{2}$ and $\beta(0,\,0)\, =\, B_{11}$ 
are non-zero. Also, $F$ satisfies the following system of equations:
\begin{align*}
F(s,\,t) & \,=\, 0, \qquad F_x(s,\,t) \,=\, 0, \qquad F_y(s,\,t) \,= \,0, \qquad \textnormal{and} 
\qquad (F_{xx} F_{yy} - F_{xy}^2)(s,\,t) \,=\, 0.
\end{align*} 
Using the first three equations, we can solve for $C_{0},\, C_{1}$ and $B_{00}$. 
Plugging that into the last equation, we can solve for $t$ in terms of $B_{10}$ using the 
implicit function theorem. Therefore, such an $F$ exists. This proves \cref{k_clsr77}. 

Now the multiplicity of the intersection will be computed. A straightforward calculation yields 
\begin{align*}
F(a_d,\, 0)\,& =\, -\frac{Q(a_d) a_d}{4 \alpha(0) \beta(0,\,0) Q(0)} B_{10} ^3 + O(B_{10}^4).
\end{align*}
Since the coefficient of $B_{10}$ is non-zero, it follows that 
the multiplicity of the intersection is $3$. This completes the proof of 
\Cref{theorem_for_many_Tks_A1F_pp_Break_cusp}.
\end{proof}

\begin{thm}
\label{theorem_for_many_Tks_A1F_pp_Break_Fin_cusp}
Let $d\,\in\, \mathbb{Z}^{\geq 3}$. The following equality of homology classes
in $\mathsf{M}_{{d+1}}^1(d)$ holds: 
\begin{align}\label{many_Tks_A1F_pp_Break_Fin_cusp}
\big[\mathsf{A}_2^{\mathsf{F}} \underbrace{\mathsf{T}_{0} \cdots
\mathsf{T}_{0}}_{d-\textnormal{times}}\mathsf{F}\big]\cdot 
\big[\big(x_{d+1} \in \mathsf{L}\big)\big]\cdot \big[\big(x_{d+1} \in \mathcal{C}_d\big)\big] &\,\,= \,\,
\sum_{i=1}^{d}\big[\mathsf{A}_2^{\mathsf{F}}
\underbrace{\mathsf{T}_{0}\cdots \mathsf{T}_{0}}_{d-\textnormal{times}} \mathsf{F}\big] \cdot \big[\big(x_i = x_{d+1}\big)\big] \nonumber \\
& + \sum_{i=1}^d 3 \big[\big(\mathsf{R}\mathsf{T}_1 \big)_{d+1}^i(d)\big] + \big[\big(\mathsf{R}\mathsf{A}_2^{\mathsf{F}}\big)_{d+1}(d)\big] 
\nonumber \\ 
& +3 \big[\big(\mathsf{R}\mathsf{A}_1^{\mathsf{L}}\big)_{d+1}(d)\big] 
+ 2 \big[\big(\mathsf{R}\mathsf{T}_2\big)_{d+1}(d)\big].
\end{align}
\end{thm}

\begin{rem}
As before, \cref{na2_d_pts} is obtained by intersecting both sides of \cref{many_Tks_A1F_pp_Break_Fin_cusp} with 
\[\mu_{d+1}\ :=\ y_1^2 y_d^{\delta_d-(d+3)} a_1 a_2 \cdots a_d a_{d+1}.\]
Finally, \cref{na1_d_minus_1_cusp,na2_sm,na2_d_pts} is proved. All of them together
produce \cref{mr_12} (that is, the second part of \Cref{mr1}).
\end{rem}

\begin{proof}[{\textbf{Proof of \Cref{theorem_for_many_Tks_A1F_pp_Break_Fin_cusp}}}]
The first term on the right-hand side of \cref{many_Tks_A1F_pp_Break_Fin_cusp} is justified as before. 
The remaining terms need to be justified. 

Start with the second term. The justification of the closure and multiplicity is similar to that of the 
third term of \cref{many_Tks_A1F_pp_Break_cusp} (see the proof of \Cref{theorem_for_many_Tks_A1F_pp_Break_cusp} 
for details).

Next, consider the third term. 
Its justification is similar to that of \cref{many_Tks_A1F_pp_Break_Fin} 
(see proof of \Cref{theorem_for_many_Tks_A1F_pp_Break_Fin}). 

To justify the fifth term on the right hand side of \cref{many_Tks_A1F_pp_Break_Fin_cusp},
pass to affine coordinates. Let
\begin{align*}
F(x,\,y)\,& =\, Q(x)(C_0)+ y(B_{00} + B_{10}x + \frac{B_{20}}{2}x^2 + \frac{\alpha(x)}{6} x^3 +
\beta(y) y + \gamma(x,\,y) xy),\\ 
 \textnormal{where \, }& Q(x)\,:=\, (x-A_1) \cdots (x-A_d), \\ 
& \alpha(x)\,:=\, B_{30} + \frac{B_{40}}{4} x + \cdots,\\
& \beta(y)\,:=\,B_{01} + \frac{B_{02}}{2} y + \cdots, \\ 
&\gamma(x,y) \,:=\, B_{11} + \frac{B_{21}}{2} x + \frac{B_{12}}{2} y + \cdots ;
\end{align*}
here $C_{0},\, B_{00} ,\, B_{10}$ and $B_{20}$ are small and
$\alpha(0)\,=\, B_{30}$ and $\beta(0) \,=\, B_{01}$ are non-zero. Again, we try to solve for 
\begin{align*}
F(s,\,t) & \,=\, 0, \quad F_x(s,\,t) \,=\, 0, \quad F_y(s,\,t) \,=\, 0, \quad \textnormal{and} 
\quad (F_{xx} F_{yy} - F_{xy}^2)(s,\,t) \,=\, 0. 
\end{align*} 
Using the first three equations, we can solve for $C_{0},\, B_{00}$ and $B_{10}$.
Plugging these into the last equation, $B_{20}$ can be solved in terms of $t$ 
using the implicit function theorem. Now evaluate $F$ at $(a_{d+1}, 0)$ and get that 
\begin{align*}
F(a_{d+1}, 0)&\, =\, -\frac{Q(a_{d+1}) \beta(0)}{Q(0)} t^2 + O(t^3). 
\end{align*}
The multiplicity of the intersection is clearly $2$. 

To justify the fourth term on the right-hand side of \cref{many_Tks_A1F_pp_Break_Fin_cusp},
switch to affine coordinates. Let
\begin{align*}
F(x,\,y)\,& =\, Q(x)(C_0)+ y(B_{00} + B_{10} x + B_{01}y + \frac{\alpha(x)}{2} x^2 +
\frac{\beta(y)}{2} y^2 + \gamma(x,\,y) xy),
\end{align*}
where $C_{0},\, B_{00} ,\, B_{10}$ and $B_{01}$ are small and $\alpha(0)\beta(0)-\gamma(0)^2$ 
is non-zero. Furthermore, assume that $\alpha(0) \,=\, B_{20} \,\neq\, 0$. 
The latter can be assumed because the $\mathtt{x}$--axis is not 
one of the branches of the node. Again, we try to solve for 
\[
F(s,\,t) \, =\, 0, \quad F_x(s,\,t) \,=\, 0, \quad F_y(s,\,t) \,=\, 0 \quad \textnormal{and} 
\quad (F_{xx} F_{yy} - F_{xy}^2)(s,t) \,=\, 0. 
\]
Using the first three equations, we can solve for $C_{0},\, B_{00}$ and $B_{10}$.
Plugging these into the last equation, $B_{01}$ can be solved in terms of $t$ 
using the implicit function theorem. 
Now, evaluate $F$ at $(a_{d+1},\, 0)$ and get that 
\begin{align*}
F(a_{d+1}, \,0)\ =\ Q(a_{d+1})\Big(\frac{\alpha(0)\beta(0)-\gamma(0)^2}{2 \alpha(0) Q(0)}\Big) t^3 + O(t^4). 
\end{align*}
The multiplicity of the intersection is clearly $3$. 
This completes the proof of \Cref{theorem_for_many_Tks_A1F_pp_Break_Fin_cusp}. 
\end{proof} 

\section{Extension of Caporaso-Harris for a family of nodal cubics} 
\label{CH_pl_nodal_details}

In this section the following question will be addressed:
\begin{center}
{\it What is $N(r,s)$, the number of nodal planar cubics in $\mathbb{P}^3$ 
that intersect $r$ generic lines and pass through $s$ generic points, where $r+2s = 11$?}
\end{center}

Notation: $N(r,\,s)\,=\, 0$ unless $r+2s \,=\,11$. 

Let $\widehat{\mathbb{P}}^3$ denote the dual of $\mathbb{P}^3$; 
it is the space of linear $\mathbb{P}^2$ inside $\mathbb{P}^3$. 
An element of $\widehat{\mathbb{P}}^3$ can be thought of as 
a non-zero linear functional $\eta\,:\, \mathbb{C}^4 \,\longrightarrow \,\mathbb{C}$ up to scaling.
Given such an $\eta$, define the projectivization of its kernel as $\mathbb{P}^2_{\eta}$. In other words, 
\[
\mathbb{P}^2_{\eta}\ :=\ \mathbb{P}(\eta^{-1}(0)). 
\]
Hence, an element of $\widehat{\mathbb{P}}^3$ determines a plane in $\mathbb{P}^3$ and vice versa. 
Define $\big(\mathbb{P}^2\big)_{\textsf{Fb}}$ as follows: 
\[
\big(\mathbb{P}^2\big)_{\textsf{Fb}}\ :=\ \{([\eta],\, q) \,\in\,
\widehat{\mathbb{P}}^3\times \mathbb{P}^3\,\big\vert\,\, \eta(q)\,=\,0\}.
\]
So $\big(\mathbb{P}^2\big)_{\textsf{Fb}}$ is a fiber bundle over $\widehat{\mathbb{P}}^3$ whose fiber over
any point $[\eta]$ is $\mathbb{P}^2_{\eta}$.

Recall from the previous section that $\mathcal{D}_1$ is the space of lines in $\mathbb{P}^2$. 
Let us now make a more pedantic notation and define 
$\mathcal{D}_1(\eta)$ to be the space of lines in $\mathbb{P}^2_{\eta}$. Similarly, define 
$\mathcal{D}_3(\eta)$ to be the space of cubics in $\mathbb{P}^2_{\eta}$.
Now define the following fiber bundles over $\widehat{\mathbb{P}}^3$: 
\begin{equation}\label{z1}
\big(\mathcal{D}_1\big)_{\textsf{Fb}}\,\longrightarrow\, \widehat{\mathbb{P}}^3\ \, \text{ and }\,\
\big(\mathcal{D}_3\big)_{\textsf{Fb}}\,\longrightarrow\, \widehat{\mathbb{P}}^3. 
\end{equation}
The fibers over any point $[\eta]$ are $\mathcal{D}_1(\eta)$ and $\mathcal{D}_3(\eta)$, respectively.
Finally, define 
\[
\Big(\mathcal{D}_1 \times \mathcal{D}_3\Big)_{\mathsf{Fb}}\ \longrightarrow\ \widehat{\mathbb{P}}^3
\]
to be the fiber product of the two fiber bundles in \eqref{z1}. Its fiber over any $[\eta]$
is $\mathcal{D}_1(\eta) \times \mathcal{D}_3(\eta)$.

Next, define 
\[
\mathfrak{M}^m_n(3)\ :=\ \Big(\mathcal{D}_1 \times \mathcal{D}_3\Big)_{\mathsf{Fb}} \times 
\big(Y^1\times \ldots \times Y^m\big)\times \big(Y_1 \times \ldots \times Y_n\big),
\]
where each $Y^i$ and $Y_j$ is a copy of $\mathbb{P}^3$. 
The hyperplane class of $Y^i$ and $Y_j$ 
will be denoted by $B_i$ and $H_j$ respectively.
The first Chern classes of the hyperplane bundles of $\big(\mathcal{D}_1\big)_{\mathsf{Fb}}$ and 
$\big(\mathcal{D}_3\big)_{\mathsf{Fb}}$ 
will be denoted by $\lambda_1$ and $\lambda_3$ respectively. The hyperplane class of $\widehat{\mathbb{P}}^3$ 
will be denoted by $a$. Since the entire discussion is for cubics, we are denoting the space by 
$\mathfrak{M}^m_n$ and not $\mathfrak{M}^m_n(3)$; this will make reading easier.

Next, define 
\begin{align}\label{fibre_bundle_parameter_space}
\big(\mathsf{M}^m_n \big)_{\textsf{Fb}}\ &\longrightarrow\ \widehat{\mathbb{P}}^3
\end{align}
to be the fiber bundle over $\widehat{\mathbb{P}}^3$ whose fiber over each point $[\eta]$ is 
$\mathsf{M}^m_n(\eta)$.

Note that on $\mathfrak{M}^m_n$, the following equality of cycles hold:
\begin{align} 
\Big[ \big(\mathsf{M}^m_n\big)_{\textsf{Fb}} \Big]\ & =\ (a+B_1) \cdot \ldots \cdot (a+B_m) \cdot (a+H_1) \cdot \ldots \cdot (a+H_n). 
\label{Fb_plane}
\end{align}
To see why this is true, denote an element of 
$\mathfrak{M}^m_n$ by
\[\Big(\Pi,\, \mathsf{L},\, \mathcal{C}_3,\, p_1, \,\ldots,\, p_m,\, x_1,\, \cdots,\, x_n\Big);\]
here $\Pi$ denotes a plane in $\mathbb{P}^3$. 
Furthermore, $\mathsf{L}$ and $\mathcal{C}_3$ denote 
a line and a cubic respectively in the plane $\Pi$. 
Whereas $\big(\mathsf{M}^m_n \big)_{\textsf{Fb}}$ is the locus of all 
\[\Big(\Pi,\, \mathsf{L},\, \mathcal{C}_3,\, p_1,\, \cdots,\, p_m,\, x_1,\, \cdots,\, x_n\Big)\] 
in $\mathfrak{M}^m_n$ such that 
\begin{itemize}
\item The points $p_1,\, \cdots,\, p_m$ lie on the plane $\Pi$; 
\item the points $x_1, \,\cdots,\, x_n$ lie on the plane $\Pi$.
\end{itemize}
Note that the condition that $p_i$ belongs to the plane $\Pi$ corresponds to 
intersecting with the class $a+B_i$. Similarly, the condition that
$x_j$ belongs to the plane $\Pi$ corresponds to intersecting with the class $a+H_j$.
This justifies \cref{Fb_plane}.

Another notation: Suppose for each variety $Z$ ismorphic to $\mathbb{P}^2$ we have a space $S(Z)$.
Then $\big(S\big)_{\mathsf{Fb}}$ is the fiber bundle over $\widehat{\mathbb{P}}^3$ whose fiber over each 
$[\eta]$ is $S(\mathbb{P}^2_\eta)$. To give an example,
recall the space $\mathsf{A}_1^{\mathsf{F}} \mathsf{T}_0$ defined in 
section \ref{CH_one_node_review}. It is the locus of all $(\mathsf{L},\, \mathcal{C}_3,\, p,\, x_1)$ in 
$\mathsf{M}^1_1$ such that 
\begin{itemize} 
\item the cubic $\mathcal{C}_3$ has a node at $p$, and 
\item The nodal point $p$ does not lie on the line $\mathsf{L}$.
\end{itemize} 
Now define $\big(\mathsf{A}_1^{\mathsf{F}} \mathsf{T}_0)_{\mathsf{Fb}}
\,\subseteq\, \big(\mathsf{M}^1_1\big)_{\mathsf{Fb}}$ to be
the locus of all $(\Pi,\, \mathsf{L},\, \mathcal{C}_3,\, p,\, x_1)$ such that 
\begin{itemize} 
\item the cubic $\mathcal{C}_3$ has a node at $p$, and
\item the nodal point $p$ does not lie on the line $\mathsf{L}$.
\end{itemize} 
Pictorially this looks like:

\begin{center}
            \begin{figure}[h]
                \centering
                \includegraphics[scale=0.28]{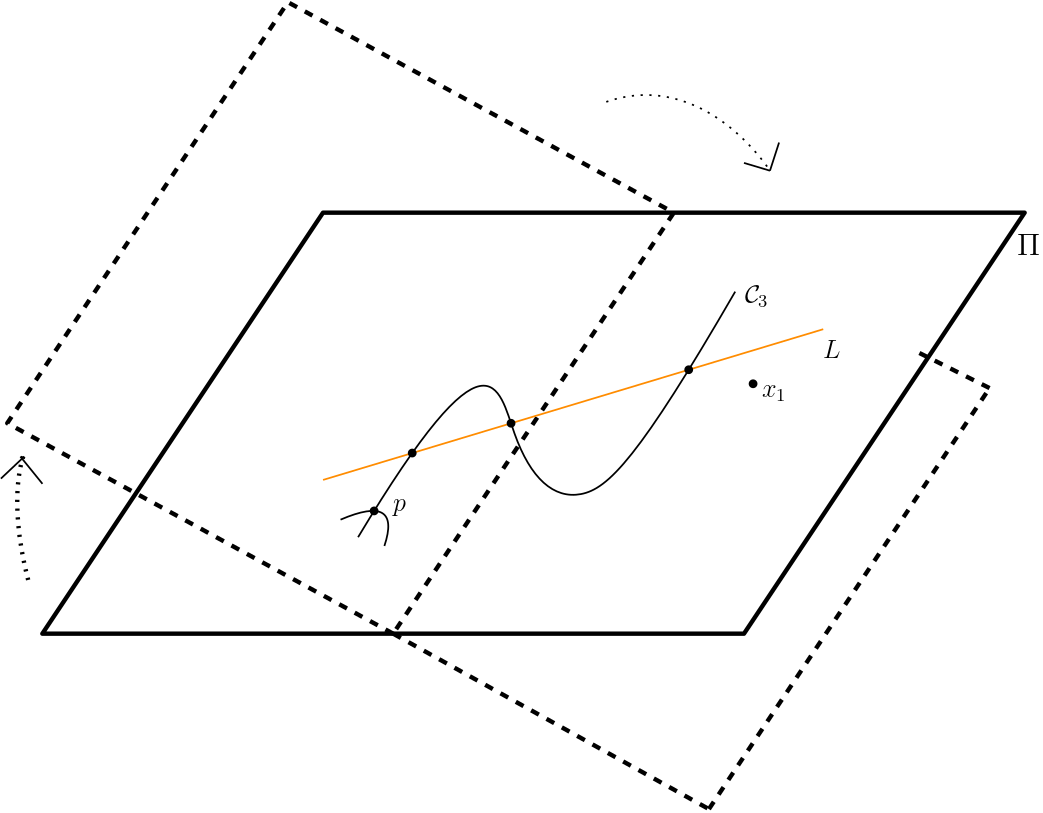}
                \caption{$\big(\mathsf{A}_1^{\mathsf{F}} \mathsf{T}_0)_{\mathsf{Fb}}$}
                \label{fbexample}
            \end{figure}
        \end{center}

All the other spaces defined in 
section \ref{CH_one_node_review} can be analogously defined.
 
Next consider the cohomology ring of $\mathfrak{M}^m_n$. 
Its generators are 
\[ a,\, \lambda_1,\, \lambda_3,\, B_1,\, \cdots,\, B_m,\, H_1,\, \cdots,\, H_n. \] 
The ring structure is given by the following relations: 
\begin{align}
\lambda_1^3\, & =\, -\lambda_1^2 a - \lambda_1 a^2 - a^3, \nonumber \\ 
\lambda_3^{10}\, & =\, -10\lambda_3^9 a - 55 \lambda_3^8 a^2 -220 \lambda_3^7 a^3, \nonumber \\ 
B_i^4\, & =\, 0\,\,\, \forall\,\, i \,=\, 1 ,\, \cdots,\, m, \quad
H_j^4\, =\, 0\,\,\,\forall\,\, j \,=\, 1 ,\, \cdots,\, n, \quad \textnormal{and} \quad a^4\,=\,0. \label{Fb_ring}
\end{align}
This fact is justified in \cite[pp.~7--8]{BMS_R_N}. We are now ready to state the main results that will be used in
enumerating the nodal planar cubics in $\mathbb{P}^3$.
The statements can be considered as family version of Theorems
\ref{step1_CH_p}, \ref{theorem_for_many_Tks_A1F_pp}, \ref{theorem_for_many_Tks_A1F_pp_Break}
and \ref{theorem_for_many_Tks_A1F_pp_Break_Fin} with $d\,=\,3$.
The spaces involved here are fiber bundle analogue of the spaces involved in the
respective statements in \Cref{CH_one_node_review}. We will use the fiber bundle version of the
spaces without defining them.  

\begin{thm} 
\label{step1_CH_p_FB}
The following equality of homology classes in $\big(\mathsf{M}^1_1\big)_{\mathsf{Fb}}$ holds:
\begin{align}
\big[\big(\mathsf{A}_1^{\mathsf{F}} \mathsf{F}\big)_{\mathsf{Fb}}\big]\cdot
\big[\big(x_1 \in \mathsf{L}\big)_{\mathsf{Fb}}\big] \cdot \big[ \big(x_1 \in \mathcal{C}_3\big)_{\mathsf{Fb}} \big] 
&\ =\ \big[\big(\mathsf{A}_1^{\mathsf{F}}\mathsf{T}_0\big)_{\mathsf{Fb}}\big]. \label{step1_CH_cycle_FB}
\end{align}
\end{thm}

This is the family version of \Cref{step1_CH_p}; 
the proof is identical.

Now the goal is to use 
\cref{step1_CH_cycle_FB} and extract numbers. Here it is more convenient to do intersection theory 
on the ambient space $\mathfrak{M}^1_1$, which we recall is given by
\[ \mathfrak{M}^1_1:= \Big(\mathcal{D}_1 \times \mathcal{D}_3\Big)_{\mathsf{Fb}} \times Y^1 \times Y_1.\]
Let $\alpha$ be a class in 
$\big(\mathsf{M}^1_1\big)_{\mathsf{Fb}}$, and let $\mu$ be a class in $\mathfrak{M}^1_1$. 
Denote the restriction of $\mu$ to $\big(\mathsf{M}^1_1\big)_{\mathsf{Fb}}$ by the same symbol $\mu$. 
Then, \cref{Fb_plane} immediately yields the following: 
\begin{equation}\label{fb_comp}
\alpha \cdot \mu \ =\ \alpha \cdot (a+B_1) \cdot (a + H_1) \cdot \mu.
\end{equation}
The left-hand side of equation \eqref{fb_comp} is an intersection number in 
$\big(\mathsf{M}^1_1\big)_{\mathsf{Fb}}$, while the right-hand side is an intersection number in 
$\mathfrak{M}^1_1$. (If two classes are not of complementary dimension, 
then the corresponding intersection number is defined to be zero.)

As before, the goal is to 
intersect both sides of \cref{step1_CH_cycle_FB} with a suitable $\mu$, and extract numbers using \cref{fb_comp}. 
The objective is to compute $N(r,\,s)$, the 
number of nodal planar cubics in $\mathbb{P}^3$ that intersect $r$ generic lines and $s$ generic points. 
Let us start by explaining how equation \eqref{step1_CH_cycle_FB} will help us compute $N(r,\,s)$. 

Similar to the space $\mathsf{A}_1^{\mathsf{F}}$ defined in \Cref{CH_one_node_review},
define the analogous fiber bundle $\big(\mathsf{A}_1^{\mathsf{F}}\big)_{\mathsf{Fb}}$ over
$\widehat{\mathbb{P}^3}$. An element of this space is a triple consisting of a line, a cubic and a point
such that all of them lie on a single plane in $\mathbb{P}^3$.

We will compute
\begin{align}
\big[\big(\mathsf{A}_1^{\mathsf{F}}\big)_{\mathsf{Fb}}\big]\cdot \lambda_1^m
\lambda_3^j a^l, \label{A1f_Fb_int_num}.
\end{align}
It will be shown that it is enough to compute \eqref{A1f_Fb_int_num}.

Using \cref{step1_CH_cycle_FB}, we will show that
\begin{align}
\big[\big(\mathsf{A}_1^{\mathsf{F}}\big)_{\mathsf{Fb}}\big]\cdot \lambda_1^m \lambda_3^j a^l\ &=\ 
\frac{1}{3} \big[\big(\mathsf{A}_1^{\mathsf{F}}\mathsf{T}_0\big)_{\mathsf{Fb}}\big] \cdot \lambda_1^m \lambda_3^j a^l.
\label{CH_FB_Step1_new_form}
\end{align} 
Hence, computation of
\begin{align}
\big[\big(\mathsf{A}_1^{\mathsf{F}} \mathsf{T}_0\big)_{\mathsf{Fb}}\big]\cdot \lambda_1^m \lambda_3^j a^l \label{a1ft0_int_num}
\end{align}
would give \cref{A1f_Fb_int_num}.

To justify these assertions, we begin by showing that $N(r,s)$ can 
be obtained as a suitable intersection number involving the class 
$\big[\big(\mathsf{A}_1^{\mathsf{F}}\big)_{\mathsf{Fb}}\big]$. The reader is reminded
that this will be an intersection number inside the space $\big(\mathsf{M}^1_0\big)_{\mathsf{Fb}}$. 

The homology classes represented 
by the subspace of cubics that intersect a generic line and intersect a generic point will be denoted by 
$\mathcal{H}^{\mathcal{C}_3}_l$ and $\mathcal{H}^{\mathcal{C}_3}_p$ respectively. 
Similarly, the homology classes represented 
by the subspace of lines that intersect a generic line and intersect a generic point will be denoted by 
$\mathcal{H}^{\mathsf{L}}_l$ and $\mathcal{H}^{\mathsf{L}}_p$
respectively. 
It is shown in \cite[pp.~10--11]{BMS_R_N} that 
\begin{align}
\mathcal{H}^{\mathsf{L}}_l &\, =\, \lambda_1 + a, \quad \mathcal{H}^{\mathsf{L}}_{p} \,=\, \lambda_1 a, \quad
\mathcal{H}^{\mathcal{C}_3}_l \,=\, \lambda_3 + 3 a \quad \textnormal{and} \quad
\mathcal{H}^{\mathcal{C}_3}_p\, =\, \lambda_3 a. \label{const_P3_pl}
\end{align}
Now observe that 
\begin{align}
N(r,\,s)\ &=\ \big[\big(\mathsf{A}_1^{\mathsf{F}}\big)_{\mathsf{Fb}}\big] \cdot 
\big(\mathcal{H}^{\mathcal{C}_3}_l\big)^r \cdot \big(\mathcal{H}^{\mathcal{C}_3}_p\big)^s \cdot 
\big(\mathcal{H}^{\mathsf{L}}_l\big)^2.\label{a1F_num_fb}
\end{align}
To see why this is so, first assume that $r+2s\,=\,11$; otherwise both the sides of \cref{a1F_num_fb} 
are equal to zero 
because of the convention (namely that if two classes are not of complementary dimensions, then their 
intersection number is zero).

Since $r+2s\,=\,11$, intersecting $\big[\big(\mathsf{A}_1^{\mathsf{F}}\big)_{\mathsf{Fb}}\big]$
with 
$\big(\mathcal{H}^{\mathcal{C}_3}_l\big)^r \cdot \big(\mathcal{H}^{\mathcal{C}_3}_p\big)^s$ fixes the 
cubic (and also the plane inside which the cubic lies).
Now intersecting with $\big(\mathcal{H}^{\mathsf{L}}_l\big)^2$ does the following: the two lines will 
intersect the plane at two points and through these two points on the plane a unique line passes. That fixes the 
line. This proves equation \eqref{a1F_num_fb}.
Using equations \eqref{a1F_num_fb}, \eqref{const_P3_pl} and \eqref{Fb_ring}, we conclude that 
$N(r,s)$ is a sum of intersection numbers of the type given by \cref{A1f_Fb_int_num}. 

It remains to prove \cref{CH_FB_Step1_new_form}. We will rewrite \cref{step1_CH_cycle_FB} 
as an equality of homology classes in $\mathfrak{M}^1_1$. 
A couple of definitions are needed for this. Define 
$\big(\mathsf{A}_1^{\mathsf{F}}\big)_{\mathsf{Fb}} \mathsf{F}$ to be the locus of points
$(\Pi,\, \mathsf{L},\, \mathcal{C}_3,\, p, \,x_1)$ in $\mathfrak{M}^1_1$ 
such that $(\Pi,\, \mathsf{L},\, \mathcal{C}_3,\, p)\,\in\, \big(\mathsf{A}_1^{\mathsf{F}}\big)_{\mathsf{Fb}}$. 
There is no restriction on the point $x_1$; it need not have to lie on the plane $\Pi$. In other words,
\begin{align}
\big(\mathsf{A}_1^{\mathsf{F}}\big)_{\mathsf{Fb}} \mathsf{F}\ & =\
\big(\mathsf{A}^{\mathsf{F}}\big)_{\mathsf{Fb}} \times Y_1. 
\label{free_prod} 
\end{align} 
Note that $\big(\mathsf{A}_1^{\mathsf{F}}\big)_{\mathsf{Fb}} \mathsf{F}$ is \textit{not} the same as 
$\big(\mathsf{A}_1^{\mathsf{F}}\mathsf{F}\big)_{\mathsf{Fb}}$; in the latter case the point $x_1$ is required 
to lie on the plane $\Pi$. Denote by $(x_1 \in \Pi)$ the 
locus of all points $(\Pi,\, \mathsf{L},\, \mathcal{C}_3,\, p,\, x_1)$ in $\mathfrak{M}^1_1$ 
such that the point $x_1$ lies on the plane $\Pi$. 

Note that \cref{step1_CH_cycle_FB} is equivalent to the fact that on 
$\mathfrak{M}^1_1$ the following equality of homology classes holds:
\begin{align}
\big[\big(\mathsf{A}_1^{\mathsf{F}}\big)_{\mathsf{Fb}} \mathsf{F}\big]\cdot \big[(x_1 \in \Pi)\big]\cdot
\big[\big(x_1 \in \mathsf{L}\big)\big] \cdot \big[ \big(x_1 \in \mathcal{C}_3\big)\big] 
&\ =\ \big[\big(\mathsf{A}_1^{\mathsf{F}}\mathsf{T}_0\big)_{\mathsf{Fb}}\big]. \label{step1_CH_cycle_FB_Free}
\end{align}
A subtle point needs explanation. Consider the left-hand side of \cref{step1_CH_cycle_FB_Free}. Strictly speaking, we have not actually defined the third and fourth 
factors on the left-hand side. We define 
$\big(x_1 \in \mathsf{L}\big)$ to be a subspace of $\mathfrak{M}^1_1$ (not $\big(\mathsf{M}^1_1\big)_{\mathsf{Fb}}$), 
where the point $x_1$ is required to lie on the line. Notice the difference between 
$\big(x_1 \in \mathsf{L}\big)$ and $\big(x_1 \in \mathsf{L}\big)_{\mathsf{Fb}}$; the latter is a subspace of 
$\big(\mathsf{M}^1_1\big)_{\mathsf{Fb}}$. However, restricted to $\big(x_1 \in \Pi\big)$, these two subspaces 
are the same. 
The same remark applies for the space $\big(x_1 \in \mathcal{C}_3\big)$. 

Note that in $\mathfrak{M}^1_1$, the following equalities of homology classes hold: 
\begin{align}
\big[x_1 \in \Pi\big] & \,=\, a+H_1, \quad\big[\big(x_1 \in \mathsf{L}\big)_{\mathsf{Fb}}\big]
\,=\, \lambda_1 + H_1
\quad \textnormal{ and } \quad \big[ \big(x_1 \in \mathcal{C}_3\big)_{\mathsf{Fb}} \big]
\,=\, \lambda_3 + 3 H_1. \label{k1}
\end{align} 
Next, using \cref{free_prod} it follows that 
\begin{align}
\big[\big(\mathsf{A}_1^{\mathsf{F}}\big)_{\mathsf{Fb}}\big]\cdot \lambda_1^m \lambda_3^j a^l &\ =\ 
\big[\big(\mathsf{A}_1^{\mathsf{F}}\big)_{\mathsf{Fb}} \mathsf{F}\big] \cdot \lambda_1^m \lambda_3^j a^l H_1^3. \label{kkiuy}
\end{align} 
Finally, the following identity is obtained:
\begin{align}
\label{div_3}
\big[\big(\mathsf{A}_1^{\mathsf{F}}\big)_{\mathsf{Fb}} \mathsf{F}\big] \cdot \lambda_1^m \lambda_3^j H_1^{n} a^l 
&\ =\ \begin{cases} \frac{1}{3} \big[\big(\mathsf{A}_1^{\mathsf{F}}\mathsf{T}_0\big)_{\mathsf{Fb}}\big] \cdot 
\lambda_1^m \lambda_3^j a^l \qquad \textnormal{if} \qquad n \,=\, 3\\ 
0 \qquad \qquad \qquad \qquad \qquad \quad \textnormal{if} \qquad n\,\neq\, 3. 
\end{cases} 
\end{align}
The second case of \cref{div_3} follows from \cref{free_prod}. The first case of \cref{div_3} follows from 
equations \eqref{k1}, 
\eqref{step1_CH_cycle_FB_Free} and \eqref{free_prod}. 
Finally, note that \cref{div_3}, combined with \cref{kkiuy} gives \cref{CH_FB_Step1_new_form}. 

To summarize what has been established so far, the goal is to compute the number $N(r,s)$.
It is enough to compute all the intersection numbers given by \cref{A1f_Fb_int_num}.
It is then established that it is enough to compute all intersection numbers given by 
\cref{a1ft0_int_num}; this is done through \cref{CH_FB_Step1_new_form}.
 
The following theorem is used in computing the intersection numbers given by \cref{a1ft0_int_num}.

\begin{thm}
\label{theorem_for_many_Tks_A1F_pp_FB}
The following equality of homology classes in $\big(\mathsf{M}^1_{2}\big)_{\mathsf{Fb}}$ holds:
\begin{align}\label{many_Tks_A1F_pp_FB}
\big[\big(\mathsf{A}_1^{\mathsf{F}} \mathsf{T}_{0}\mathsf{F}\big)_{\mathsf{Fb}}\big]\cdot 
\big[\big(x_{2}\in \mathsf{L}\big)_{\mathsf{Fb}}\big]\cdot 
\big[\big(x_{2}\in \mathcal{C}_3\big)_{\mathsf{Fb}}\big] \,\,&= \,\,
\big[\big(\mathsf{A}_1^{\mathsf{F}} \mathsf{T}_{0}\mathsf{T}_{0}\big)_{\mathsf{Fb}}\big]\nonumber \\ 
& + \big[\big(\mathsf{A}_1^{\mathsf{F}}
\mathsf{T}_{0}\mathsf{T}_{0} \mathsf{F}\big)_{\mathsf{Fb}}\big] \cdot \big[\big(x_1 = x_{2}\big)_{\mathsf{Fb}}\big].
\end{align}
\end{thm}

This is the family version of \Cref{theorem_for_many_Tks_A1F_pp_FB} when $n=1$; the proof is identical.

Let us now explain how this Theorem will help us compute the intersection numbers given by \cref{a1ft0_int_num}. 
Note that 
\begin{align}
\big[\big(\mathsf{A}_1^{\mathsf{F}}\mathsf{T}_0\big)_{\mathsf{Fb}}\big] \cdot \lambda_1^m \lambda_3^j a^l\ & =\ 
\big[\big(\mathsf{A}_1^{\mathsf{F}}\mathsf{T}_0\big)_{\mathsf{Fb}} \mathsf{F}\big] \cdot 
\lambda_1^m \lambda_3^j a^l H_2^3, \qquad \nonumber \\ 
\textnormal{where} \qquad \big(\mathsf{A}_1^{\mathsf{F}}\mathsf{T}_0\big)_{\mathsf{Fb}} \mathsf{F} \
& =\ \big(\mathsf{A}_1^{\mathsf{F}}\mathsf{T}_0\big)_{\mathsf{Fb}}\times Y_2.\label{k2}
\end{align}
The following will be shown:
\begin{align}
\label{div_4}
\big[\big(\mathsf{A}_1^{\mathsf{F}} \mathsf{T}_0\big)_{\mathsf{Fb}} \mathsf{F}\big] \cdot \lambda_1^m \lambda_3^j H_2^{n} a^l 
\ & =\ \begin{cases} \frac{1}{2} \big[\big(\mathsf{A}_1^{\mathsf{F}}\mathsf{T}_0 \mathsf{T}_0\big)_{\mathsf{Fb}}\big] \cdot 
\lambda_1^m \lambda_3^j a^l \qquad \textnormal{if} \qquad n\,=\, 3\\ 
0 \qquad \qquad \qquad \qquad \qquad \quad \textnormal{if} \qquad n\,\neq\, 3.
\end{cases} 
\end{align}
The second case of \cref{div_4} follows immediately 
from the definition of 
$\big(\mathsf{A}_1^{\mathsf{F}}\mathsf{T}_0\big)_{\mathsf{Fb}} \mathsf{F}$. 
Let us now justify the first case, namely when $n\,=\,3$. 
Note that \cref{many_Tks_A1F_pp_FB} 
is equivalent to the fact that on $\mathfrak{M}^1_2$ the following equality of homology classes hold:
\begin{align}\label{many_Tks_A1F_pp_FB_Free}
\big[\big(\mathsf{A}_1^{\mathsf{F}} \mathsf{T}_{0}\big)_{\mathsf{Fb}} \mathsf{F}\big]\cdot 
\big[\big(x_2 \in \Pi\big)\big]\cdot 
\big[\big(x_{2}\in \mathsf{L}\big)\big]\cdot 
\big[\big(x_{2}\in \mathcal{C}_3\big)\big] \,\,&= \,\,
\big[\big(\mathsf{A}_1^{\mathsf{F}} \mathsf{T}_{0}\mathsf{T}_{0}\big)_{\mathsf{Fb}}\big]\nonumber \\ 
& + \big[\big(\mathsf{A}_1^{\mathsf{F}}
\mathsf{T}_{0}\mathsf{T}_{0}\big)_{\mathsf{Fb}}\mathsf{F}\big] \cdot \big[\big(x_1 = x_{2}\big)\big].
\end{align}
Now note that in $\mathfrak{M}^1_2$, the following equality of homology classes holds:
\begin{align}
\big[\big(x_2 \in \Pi\big)\big]\, =\, a+H_2, \quad \big[\big(x_2 \in \mathsf{L}\big)\big]\, & =\, \lambda_1 + H_2,
\quad \big[ \big(x_2 \in \mathcal{C}_3\big) \big] \,=\, \lambda_3 + 3 H_2 \nonumber\\ 
\textnormal{and } \quad \big[\big(x_1 = x_{2}\big)\big] \,& =\, H_1^3 + H_1^2 H_2 + H_1 H_2^2 + H_2^3.\label{k8}
\end{align} 
The first case of \cref{div_4} now follows immediately from equations \eqref{many_Tks_A1F_pp_FB_Free} 
and \eqref{k8}. 

Using equations \eqref{div_4} and \eqref{k2} it follows
that in order to compute the intersection numbers given by \cref{a1ft0_int_num}, 
it is enough to compute the following intersection number:
\begin{align}
\big[\big(\mathsf{A}_1^{\mathsf{F}}\mathsf{T}_0 \mathsf{T}_0\big)_{\mathsf{Fb}}\big] \cdot \lambda_1^m \lambda_3^j a^l. 
\label{a1ft0_int_num_st2}
\end{align}
it was also shown that to compute $N(r,s)$, it is sufficient to compute 
the numbers given by \cref{a1ft0_int_num}. 
Hence, to compute 
$N(r,s)$, it suffices to compute the numbers given by \cref{a1ft0_int_num_st2}.
The following theorem will be used for it.

\begin{thm}
\label{theorem_for_many_Tks_A1F_pp_Break_FB}
The following equality of homology classes in $\big(\mathsf{M}_{{3}}^1\big)_{\mathsf{Fb}}$ 
holds:
\begin{align}\label{many_Tks_A1F_pp_Break_FB}
\big[\big(\mathsf{A}_1^{\mathsf{F}} \mathsf{T}_{0} \mathsf{T}_{0}\mathsf{F}\big)_{\mathsf{Fb}}\big]\cdot 
\big[\big(x_3 \in \mathsf{L}\big)_{\mathsf{Fb}}\big]\cdot \big[\big(x_3 \in \mathcal{C}_3\big)_{\mathsf{Fb}}\big] \,\,&= \,\,
\big[\big(\mathsf{A}_1^{\mathsf{F}} \mathsf{T}_{0} \mathsf{T}_0 \mathsf{T}_{0}\big)_{\mathsf{Fb}}\big] \nonumber \\ 
& + \big[\big(\mathsf{A}_1^{\mathsf{F}}
\mathsf{T}_{0}\mathsf{T}_0\mathsf{F}\big)_{\mathsf{Fb}}\big] \cdot \big[\big(x_1 = x_3\big)_{\mathsf{Fb}}\big] \nonumber \\ 
& + \big[\big(\mathsf{A}_1^{\mathsf{F}}
\mathsf{T}_{0}\mathsf{T}_0\mathsf{F}\big)_{\mathsf{Fb}}\big] \cdot \big[\big(x_2 = x_3\big)_{\mathsf{Fb}}\big] \nonumber \\ 
& + \big[\big(\mathsf{R}_3\big)_{\mathsf{Fb}}\big].
\end{align}
\end{thm}
This the family version of \Cref{theorem_for_many_Tks_A1F_pp_Break}; the proof is identical. 

To explain how this helps in computing the intersection numbers in \cref{a1ft0_int_num_st2}, 
note that 
\begin{align}
\big[\big(\mathsf{A}_1^{\mathsf{F}}\mathsf{T}_0\mathsf{T}_0\big)_{\mathsf{Fb}}\big] \cdot \lambda_1^m \lambda_3^j a^l & = 
\big[\big(\mathsf{A}_1^{\mathsf{F}}\mathsf{T}_0 \mathsf{T}_0\big)_{\mathsf{Fb}} \mathsf{F}\big] \cdot 
\lambda_1^m \lambda_3^j a^l H_3^3, \nonumber \\ 
\textnormal{where} \qquad \big(\mathsf{A}_1^{\mathsf{F}}\mathsf{T}_0 \mathsf{T}_0\big)_{\mathsf{Fb}} \mathsf{F} 
& = \big(\mathsf{A}_1^{\mathsf{F}}\mathsf{T}_0 \mathsf{T}_0\big)_{\mathsf{Fb}} \times Y_3.\nonumber
\end{align}
It will be shown that 
\begin{align}
\label{div_5}
\big[\big(\mathsf{A}_1^{\mathsf{F}} \mathsf{T}_0 \mathsf{T}_0\big)_{\mathsf{Fb}} \mathsf{F}\big] \cdot \lambda_1^m \lambda_3^j H_3^{n} a^l 
& = \begin{cases} \big[\big(\mathsf{A}_1^{\mathsf{F}}\mathsf{T}_0 \mathsf{T}_0\mathsf{T}_0\big)_{\mathsf{Fb}}\big] \cdot 
\lambda_1^m \lambda_3^j a^l \\ 
\qquad \quad +\big[\big(\mathsf{R}_3\big)_{\mathsf{Fb}}\big] \cdot \lambda_1^m \lambda_3^j a^l
\qquad \textnormal{if} \qquad n\,=\, 3\\ 
0 \qquad \qquad \qquad \qquad \qquad \quad \textnormal{if} \qquad n\,\neq\, 3. 
\end{cases} 
\end{align}
The second case of \cref{div_5} is immediate from the definition of 
$\big(\mathsf{A}_1^{\mathsf{F}}\mathsf{T}_0 \mathsf{T}_0\big)_{\mathsf{Fb}} \mathsf{F}$. 
To justify the first case, namely when $n\,=\,3$, note that \cref{many_Tks_A1F_pp_Break_FB} 
is equivalent to the fact that on $\mathfrak{M}^1_3$ the following equality of homology classes holds:
\begin{align}\label{many_Tks_A1F_pp_FB_Free_st3}
\big[\big(\mathsf{A}_1^{\mathsf{F}} \mathsf{T}_{0} \mathsf{T}_0\big)_{\mathsf{Fb}} \mathsf{F}\big]\cdot 
\big[\big(x_3 \in \Pi\big)\big]\cdot 
\big[\big(x_{3}\in \mathsf{L}\big)\big]\cdot 
\big[\big(x_{3}\in \mathcal{C}_3\big)\big] \,\,&= \,\,
\big[\big(\mathsf{A}_1^{\mathsf{F}} \mathsf{T}_{0}\mathsf{T}_{0}\mathsf{T}_0\big)_{\mathsf{Fb}}\big]\nonumber \\ 
& + \big[\big(\mathsf{A}_1^{\mathsf{F}}
\mathsf{T}_{0}\mathsf{T}_{0} \big)_{\mathsf{Fb}}\mathsf{F}\big] \cdot \big[\big(x_1 = x_{3}\big)\big] \nonumber \\ 
& + \big[\big(\mathsf{A}_1^{\mathsf{F}}
\mathsf{T}_{0}\mathsf{T}_{0} \big)_{\mathsf{Fb}}\mathsf{F}\big] \cdot \big[\big(x_2 = x_{3}\big)\big] \nonumber \\ 
& + \big[\big(\mathsf{R}_3\big)_{\mathsf{Fb}}\big]. 
\end{align}
Note that in $\mathfrak{M}^1_3$, the following equalities of homology classes hold:
\begin{align}
\big[\big(x_3 \in \Pi\big)\big] \,=\, a+H_3, \qquad  \big[\big(x_3 \in \mathsf{L}\big)\big]  & \,=\, \lambda_1 + H_3,
\qquad \big[ \big(x_3 \in \mathcal{C}_3\big) \big] \,=\, \lambda_3 + 3 H_3, \nonumber\\ 
\big[\big(x_1 = x_{3}\big)\big] & \,=\, H_1^3 + H_1^2 H_2 + H_1 H_2^2 + H_2^3 \qquad 
\textnormal{and} \nonumber \\ 
 \qquad\big[\big(x_2 = x_{3}\big)\big]  & \,=\, H_2^3 + H_2^2 H_3 + H_2 H_3^2 + H_3^3. \label{k8_st3}
\end{align} 
The first case of 
\cref{div_5} now follows immediately from equations \eqref{many_Tks_A1F_pp_FB_Free_st3} 
and \eqref{k8_st3}. 
Hence, we conclude, using \cref{div_5}, 
that in order to compute the 
intersection numbers given by \cref{a1ft0_int_num_st2}, it is sufficient to 
compute the following 
intersection numbers 
\begin{align*}
\big[\big(\mathsf{A}_1^{\mathsf{F}}\mathsf{T}_0 \mathsf{T}_0 \mathsf{T}_0\big)_{\mathsf{Fb}}\big] \cdot \lambda_1^m \lambda_3^j a^l 
\qquad \textnormal{and} \qquad \big[\big(\mathsf{R}_3\big)_{\mathsf{Fb}}\big] \cdot \lambda_1^m \lambda_3^j a^l. 
\end{align*}
We explain in \Cref{conics_computation} how to compute 
$\big[\big(\mathsf{R}_3\big)_{\mathsf{Fb}}\big] \cdot \lambda_1^m \lambda_3^j a^l$. Assuming that
those numbers can be computed, we conclude that in order to compute $N(r,s)$, it is
enough to compute 
\begin{align}
\label{st3_int_num}
\big[\big(\mathsf{A}_1^{\mathsf{F}}\mathsf{T}_0 \mathsf{T}_0 \mathsf{T}_0\big)_{\mathsf{Fb}}\big] \cdot \lambda_1^m \lambda_3^j a^l. 
\end{align}
The next theorem is used for it.

\begin{thm}
\label{theor_CH_FB_Step4}
The following equality of homology classes 
in $\big(\mathsf{M}_{{4}}^1\big)_{\mathsf{Fb}}$ holds:
\begin{align}\label{theor_CH_FB_Step4_eqn}
\big[\big(\mathsf{A}_1^{\mathsf{F}} \mathsf{T}_{0} \mathsf{T}_{0}\mathsf{T}_0\mathsf{F}\big)_{\mathsf{Fb}}\big]\cdot 
\big[\big(x_4 \in \mathsf{L}\big)_{\mathsf{Fb}}\big]\cdot \big[\big(x_4 \in \mathcal{C}_3\big)_{\mathsf{Fb}}\big] \,\,&= \,\,
\big[\big(\mathsf{A}_1^{\mathsf{F}}
\mathsf{T}_{0}\mathsf{T}_0\mathsf{T}_0\mathsf{F}\big)_{\mathsf{Fb}}\big] \cdot \big[\big(x_1 = x_4\big)_{\mathsf{Fb}}\big] \nonumber \\ 
& + \, \, \big[\big(\mathsf{A}_1^{\mathsf{F}}
\mathsf{T}_{0}\mathsf{T}_0\mathsf{T}_0\mathsf{F}\big)_{\mathsf{Fb}}\big] \cdot \big[\big(x_2 = x_4\big)_{\mathsf{Fb}}\big] \nonumber \\
& + \, \, \big[\big(\mathsf{A}_1^{\mathsf{F}}
\mathsf{T}_{0}\mathsf{T}_0\mathsf{T}_0\mathsf{F}\big)_{\mathsf{Fb}}\big] \cdot \big[\big(x_3 = x_4\big)_{\mathsf{Fb}}\big] \nonumber \\ 
&+ \, \, \big[\big(\mathsf{R}^1_4\big)_{\mathsf{Fb}}\big] + \big[\big(\mathsf{R}^2_4\big)_{\mathsf{Fb}}\big] 
+ \big[\big(\mathsf{R}^3_4\big)_{\mathsf{Fb}}\big] \nonumber \\ 
&+\, \, \big[ \big(\big(\mathsf{R}\mathsf{A}_1^{\mathsf{F}}\big)_{4}\big)_{\mathsf{Fb}} \big] 
+ 2 \big[ \big(\big(\mathsf{R}\mathsf{T}_1\big)_{4}\big)_{\mathsf{Fb}} \big]. 
\end{align}
\end{thm}
This is the family version of \Cref{theorem_for_many_Tks_A1F_pp_Break_Fin}; the proof is identical.

In contrast to the earlier cases, 
it is now a more non-trivial task to see how this theorem helps us compute the 
intersection numbers given by \cref{st3_int_num}.
We start as usual by observing that \cref{theor_CH_FB_Step4_eqn} 
is equivalent to the fact that on $\mathfrak{M}^1_4$, the following equality of homology classes holds:
\begin{align}\label{y8}
\big[\big(\mathsf{A}_1^{\mathsf{F}} \mathsf{T}_{0} \mathsf{T}_0 \mathsf{T}_0\big)_{\mathsf{Fb}} \mathsf{F}\big]\cdot 
\big[x_4 \in \Pi\big]\cdot 
\big[\big(x_{4}\in \mathsf{L}\big)\big]\cdot 
\big[\big(x_{4}\in \mathcal{C}_3\big)\big] \,\,&= \,\,
\big[\big(\mathsf{A}_1^{\mathsf{F}}
\mathsf{T}_{0}\mathsf{T}_{0} \mathsf{T}_0\big)_{\mathsf{Fb}}\mathsf{F}\big] \cdot \big[\big(x_1 = x_{4}\big)\big] \nonumber \\ 
& + \big[\big(\mathsf{A}_1^{\mathsf{F}}
\mathsf{T}_{0}\mathsf{T}_{0} \mathsf{T}_0\big)_{\mathsf{Fb}}\mathsf{F}\big] \cdot \big[\big(x_2 = x_{4}\big)\big] \nonumber \\ 
& + \big[\big(\mathsf{A}_1^{\mathsf{F}}
\mathsf{T}_{0}\mathsf{T}_{0} \mathsf{T}_0\big)_{\mathsf{Fb}}\mathsf{F}\big] \cdot \big[\big(x_3 = x_{4}\big)\big] \nonumber \\
&+ \, \, \big[\big(\mathsf{R}^1_4\big)_{\mathsf{Fb}}\big] + \big[\big(\mathsf{R}^2_4\big)_{\mathsf{Fb}}\big] 
+ \big[\big(\mathsf{R}^3_4\big)_{\mathsf{Fb}}\big] \nonumber \\ 
&+\, \, \big[ \big(\big(\mathsf{R}\mathsf{A}_1^{\mathsf{F}}\big)_{4}\big)_{\mathsf{Fb}} \big] 
+ 2 \big[ \big(\big(\mathsf{R}\mathsf{T}_1\big)_{4}\big)_{\mathsf{Fb}} \big]. 
\end{align}
Next note that in $\mathfrak{M}^1_4$, the following equalities of homology classes hold: 
\begin{align}
\big[x_4 \in \Pi\big] \,=\, a+H_4, \quad \big[\big(x_4 \in \mathsf{L}\big)\big] \, & =\, \lambda_1 + H_4,
\quad \big[ \big(x_4 \in \mathcal{C}_3\big) \big] \,=\, \lambda_3 + 3 H_4, \nonumber\\ 
\big[\big(x_i = x_{4}\big)\big]\, & =\, H_i^3 + H_i^2 H_4 + H_i H_4^2 + H_4^3, \quad i
\,=\, 1,\,2 ,\,3. \label{k8_st3_st4}
\end{align} 
To compute \cref{st3_int_num}, we first need to compute the following intersection numbers:
\begin{align}
\big[\big(\mathsf{R}^1_4\big)_{\mathsf{Fb}}\big] \cdot \lambda_1^m \lambda_3^j a^l, 
\qquad \big[\big(\mathsf{R}^2_4\big)_{\mathsf{Fb}}\big] \cdot \lambda_1^m \lambda_3^j a^l,& 
\qquad \big[\big(\mathsf{R}^3_4\big)_{\mathsf{Fb}}\big] \cdot \lambda_1^m \lambda_3^j a^l, \nonumber \\ 
\big[ \big(\big(\mathsf{R}\mathsf{A}_1^{\mathsf{F}}\big)_{4}\big)_{\mathsf{Fb}} \big] \cdot \lambda_1^m \lambda_3^j a^l 
\qquad \textnormal{and}& \qquad \big[ \big(\big(\mathsf{R}\mathsf{T}_1\big)_{4}\big)_{\mathsf{Fb}} \big] \cdot \lambda_1^m \lambda_3^j a^l. 
\label{conics_FB_comp2}
\end{align}
These numbers are computed in \Cref{conics_computation}. 
For now, let us assume that the intersection numbers given in \cref{conics_FB_comp2} 
can be computed.

Next, define 
\begin{align}
\mu &:= 
\lambda_1^m \lambda_3^{13-(m+n_1+n_2+n_3+n_4+l)} \cdot 
H_1^{n_1} \cdot H_2^{n_2} \cdot H_3^{n_3} \cdot H_4^{n_4} a^l, \label{mu_defn}\\
T(m, n_1, n_2, n_3, n_4, l) &:= \, \, \Big(\big[\big(\mathsf{R}^1_4\big)_{\mathsf{Fb}}\big] + \big[\big(\mathsf{R}^2_4\big)_{\mathsf{Fb}}\big] 
+ \big[\big(\mathsf{R}^3_4\big)_{\mathsf{Fb}}\big] \nonumber \\ 
&\, \, \, \, \, +\big[ \big(\big(\mathsf{R}\mathsf{A}_1^{\mathsf{F}}\big)_{4}\big)_{\mathsf{Fb}} \big] 
+ 2 \big[ \big(\big(\mathsf{R}\mathsf{T}_1\big)_{4}\big)_{\mathsf{Fb}} \big]\Big) \cdot \mu, 
\qquad \textnormal{and} \label{T_defn} \\ 
\nu &:= \lambda_1^m \lambda_3^{16-(m+n_1+n_2+n_3+n_4+l)} \cdot 
H_1^{n_1} \cdot H_2^{n_2} \cdot H_3^{n_3} \cdot H_4^{n_4} a^l, \label{nu_defn}\\ 
\varphi(m, n_1, n_2, n_3, n_4, l) &:= \big[\big(\mathsf{A}_1^{\mathsf{F}} \mathsf{T}_{0} \mathsf{T}_0 \mathsf{T}_0\big)_{\mathsf{Fb}} \mathsf{F}\big]\cdot \nu. \label{phi_defn}
\end{align}
Intersect both sides of
\cref{y8} with the class $\mu$ (as defined by \cref{mu_defn}). It
follows using equations \eqref{k8_st3_st4}, \eqref{T_defn} and \eqref{phi_defn}
that 
\begin{align}
-\varphi(m, n_1+3, n_2, n_3, n_4, l)& -\varphi(m, n_1, n_2+3, n_3, n_4, l)\nonumber \\
-\varphi(m, n_1, n_2, n_3+3, n_4, l)
&-\varphi(m, n_1+2, n_2, n_3, n_4+1, l)\nonumber \\ 
-\varphi(m, n_1, n_2+2, n_3, n_4+1, l) & -\varphi(m, n_1, n_2, n_3+2, n_4+1, l) \nonumber \\ 
+ 3 \varphi(m, n_1, n_2, n_3, n_4+2, l+1) &- \varphi(m, n_1+1, n_2, n_3, n_4+2, l)     \nonumber \\ 
 -\varphi(m, n_1, n_2+1, n_3, n_4+2, l)  
& -\varphi(m, n_1, n_2, n_3+1, n_4+2, l)  \nonumber \\ 
+ 3 \varphi(m+1, n_1, n_2, n_3, n_4+1, l+1) & + 3\varphi(m+1, n_1, n_2, n_3, n_4+2, l)  \nonumber \\ 
+ \varphi(m, n_1, n_2, n_3, n_4+1, l+1) & + \varphi(m, n_1, n_2, n_3, n_4+2, l)  \nonumber \\ 
+ \varphi(m+1, n_1, n_2, n_3, n_4, l+1) & + \varphi(m+1, n_1, n_2, n_3, n_4+1, l)  = T(m, n_1, n_2, n_3, n_4, l). \label{ms0}
\end{align}
We will now show that the quantity $\varphi(m,\, n_1,\, n_2,\, n_3,\, n_4,\, l)$ 
is computable in terms of the function $T$. 
Write down \cref{ms0} three times; once with $n_4\,=\,1$, then with $n_4\,=\,2$ and finally
with $n_4\,=\,3$. We will refer to these as equations one, two and three respectively. 
Now do the following: equation one minus $3$ times equation two 
plus $9$ times equation three. After simplification, we get 
\begin{align}
-\varphi(m, n_1+1, n_2, n_3, 3, l) & -\varphi(m, n_1, n_2+1, n_3, 3, l) \nonumber \\ 
-\varphi(m, n_1, n_2, n_3+1, 3, l) & + \varphi(m, n_1, n_2, n_3, 3, l) \nonumber \\
+3\varphi(m, n_1+2, n_2, n_3, 3, l) & + 3\varphi(m, n_1, n_2+2, n_3, n_4, l) \nonumber \\ 
+3\varphi(m, n_1+3, n_2, n_3, 3, l) &-9\varphi(m, n_1, n_2+3, n_3, 3, l) \nonumber \\ 
-9\varphi(m, n_1, n_2, n_3+3, 3, l) & = T(m, n_1, n_2, n_3, 1, l) \nonumber \\ 
& -3T(m, n_1, n_2, n_3, 2, l) + 9 T(m, n_1, n_2, n_3, 3, l). \label{mainrec}
\end{align}
Next, note that 
\begin{align}
\varphi(m, n_1, n_2, n_3, n_4, l)\ & =\ \begin{cases} 
0 \qquad \textnormal{if} \quad n_4\neq 3,\\
0 \qquad \textnormal{if} \quad n_1, n_2 ~~\textnormal{or} ~~n_3 \geq 4.  
\end{cases} \label{trivcase}
\end{align}
Finally, we note that 
\begin{align} 
\varphi(m, n_1, n_2, n_3, n_4, l)\ & =\ \varphi(m, \sigma(n_1), \sigma(n_2), \sigma(n_3), n_4, l), \label{symm}
\end{align}
where $\sigma$ is any element of the permutation group $S_3$.

Using all these information, we can now solve the recursion in a brute-force manner.
Start with the following equality, which is obtained by putting $n_1\,=\,n_2\,=\,n_3\,=\,3$ in \cref{mainrec},
and using the fact that $\varphi$ vanishes if $n_i \,\geq\, 4$ for $i\, \in\, \{1,\,2,\,3\}$:
\begin{align*}
\varphi(m,3,3,3,3,l) \,=\, T(m,3,3,3,1,l) - 3T(m,3,3,3,2,l) + 9T(m,3,3,3,3,l)
\end{align*}
Since the right-hand side is only about conics, those numbers can be computed and hence
$\varphi(m,3,3,3,3,l)$ can be computed. Using this, $\varphi(m,2,3,3,3,l)$ can be computed by putting
$n_1\,=\,2,\, n_2\,=\,n_3\,=\,3$ in \cref{mainrec}:
\begin{align*}
    \varphi(m,2,3,3,3,l) \,=\, \varphi(m,3,3,3,3,l) + T(m,2,3,3,1,l) - 3T(m,2,3,3,2,l) + 9T(m,2,3,3,3,l).
\end{align*}
We can continue in this manner and compute $\varphi(m,n_1,3,3,3,l)$ for all values
of $n_1$ (again, keeping in mind that $\varphi$ is possibly non-trivial only for values $n_1
\,\in\, \{0,\,1,\,2,\,3\}$).

We would then want to compute $\varphi(m,3,2,3,3,l)$, but by the symmetry property \cref{symm}, that is the 
same as $\varphi(m,2,3,3,3,l)$ which has been computed. Similarly, the numbers $\varphi(m,3,n_2,3,3,l)$ 
are obtained for any $n_2$ by using symmetry.

Next, compute $\varphi(m,2,2,3,3,l)$ by putting $n_1\,=\,2,\, n_2\,=\,2,\, n_3\,=\,3$ in \cref{mainrec}:
\[
\varphi(m,2,2,3,3,l)\ = \ \varphi(m,3,2,3,3,l) + \varphi(m,2,3,3,3,l) + T(m,2,2,3,1,l)
\]
$$
 - 3T(m,2,2,3,2,l) + 9T(m,2,2,3,3,l).
$$
Proceeding similarly, $\varphi(m,n_1,n_2,3,3,l)$ can be computed for all values of $n_1,\,n_2$.

We can continue this process, i.e., put in the required values of $n_i$ in \cref{mainrec}, and use 
\cref{trivcase}, \cref{symm}, as and when required. This process will end since there are only a finite 
number of steps going from $n_i\,=\,3$ to $n_i\,=\,0$. This solves the recursion problem.

\section{Computation of the characteristic number of conics in $\mathbb{P}^3$} 
\label{conics_computation}

The aim is to compute the characteristic number of conics in $\mathbb{P}^3$ that were needed to complete the 
computation in the last section. We will start by computing the characteristic number of smooth conics.
The intersection space being considered is 
\begin{align*} 
\mathfrak{M}^1_n(2)\ =\ \big(\mathcal{D}_1 \times \mathcal{D}_2\big)_{\mathsf{Fb}}
\times Y^1 \times Y_1 \times \cdots \times Y_n,
\end{align*}
where each $Y^i$ and $Y_j$ is a copy of $\mathbb{P}^3$. 
As before, the hyperplane class of $Y^i$ and $Y_j$ are denoted by $B_i$ and $H_j$ respectively.
The first Chern classes of the hyperplane bundles of $\big(\mathcal{D}_1\big)_{\mathsf{Fb}}$ and 
$\big(\mathcal{D}_2\big)_{\mathsf{Fb}}$ 
are denoted by $\lambda_1$ and $\lambda_2$, respectively, and the hyperplane class of $\widehat{\mathbb{P}}^3$ 
is denoted by $a$. Similar to \cref{fibre_bundle_parameter_space}, we can have the fiber bundle
\begin{align*}
\big(\mathsf{M}^1_n (2)\big)_{\textsf{Fb}}&\longrightarrow \widehat{\mathbb{P}}^3,
\end{align*}
where an element of
the bundle is a tuple $$\left( \Pi,\, \mathsf{L},\, \mathcal{C}_2,\, p, x_1,\, \cdots,\, x_n\right)$$
such that $\mathsf{L},\,\mathcal{C}_2$ and all the points lie on $\Pi$.
A similar argument to that of \cref{Fb_plane} shows that the cycle of
$\left[ \left(\mathsf{M}^1_n(2)\right)_{\textsf{Fb}} \right] $ inside $\mathfrak{M}^1_n(2)$ is given by
\begin{align*} 
\left[ \left(\mathsf{M}^1_n(2)\right)_{\textsf{Fb}} \right] &\ =\ (a+B_1) \cdot (a+H_1) \cdot \cdots \cdot (a+H_n). 
\end{align*}
 Note that the spaces involving conics in the previous section are actually cubics. Indeed, they are given by a conic and the prescribed line. Thus, we can think $\left(\mathcal{D}_1 \times \mathcal{D}_2\right)_{\mathsf{Fb}}$ as a subspace of
$\left(\mathcal{D}_3\right)_{\mathsf{Fb}}$. Thus restricted to $\left(\mathcal{D}_1 \times \mathcal{D}_2\right)_{\mathsf{Fb}}$, we have 
 \begin{align} \label{Chern_class_sum}
\lambda_3 &\ = \lambda_1 + \lambda_2.
\end{align}
Following the discussion in 
\cite[pp.~7--8]{BMS_R_N}, we conclude that the
cohomology ring of $\mathfrak{M}^1_n(2)$ is generated by $a,\, \lambda_1,\, \lambda_2,\, B_1,\,
\cdots,\, B_m,\, H_1,\, \cdots,\, H_n$, together with the relations 
\begin{align}
\lambda_1^3\, & =\, -\lambda_1^2 a - \lambda_1 a^2 - a^3, \nonumber \\ 
\lambda_2^6 \, & =\, -4\lambda_2^5a-10\lambda_2^4a^2-20\lambda_2^3a^3, \nonumber \\ 
B_i^4\, & =\, 0\,\,\, \forall\,\, i \,=\, 1 ,\, \cdots,\, m, \quad
H_j^4\, =\, 0\,\,\, \forall\,\, j \,=\, 1 ,\, \cdots,\, n, \quad \textnormal{and} \quad a^4\,=\,0. \label{Fb_ring_ag}
\end{align}

\subsection{Smooth conics}

Let us compute the following intersection numbers:
\begin{align}\label{conic_R_3}
\big[\big(\mathsf{R}_3\big)_{\mathsf{Fb}}\big]\cdot \lambda_1^m \lambda_3^j a^l.
\end{align} 
Recall that $\left(\mathsf{R}_3\right)_{\mathsf{Fb}}$ is the fiber bundle analogue of
$\mathsf{R}_3(3)$ (cf. \cref{definition_R_n}). It can be thought of as a subbundle of
$\left(\mathsf{M}^1_n(2)\right)_{\textsf{Fb}}$. If
$$\big(\Pi,\, \mathsf{L},\, \mathcal{C}_2,\, p,\, x_1,\, x_2,\,x_3\big)\,\in\, \big(\mathsf{R}_3\big)_{\mathsf{Fb}},$$
then $p$ lies on $\Pi$, $\mathsf{L}$ and $\mathcal{C}_2$ while $x_i$'s lie on $\mathsf{L}$. The incidences
correspond to the cycles $a+B_1$, $\lambda_1+B_1$, $\lambda_2+2B_1$ and $\lambda_1+H_i$ respectively.
The class of $\big(\mathsf{R}_3\big)_{\mathsf{Fb}}$  inside $\mathfrak{M}^1_n(2)$ is given by
\[
\big[\big(\mathsf{R}_3\big)_{\mathsf{Fb}}\big]\ = \ (a+B_1)\cdot (\lambda_1 + B_1) 
\cdot (\lambda_2 + 2 B_1)\cdot \prod_{i=1}^3\left[(a+H_i) (\lambda_1 + H_i)\right].        
\]
Intersection numbers of \cref{conic_R_3} can now be computed using \cref{Chern_class_sum}.

Next, let us consider $\left[\left(\mathsf{R}^1_4\right)_{\mathsf{Fb}}\right]$. It is a class inside $\mathfrak{M}^1_4(2)$. The cycle $[p=x_1]$ is given by 
$$H_1^3 + H_1^2 B_1 + H_1 B_1^2 + B_1^3.$$
A similar argument would yield that the cycle $\left[\left(\mathsf{R}^1_4\right)_{\mathsf{Fb}}\right]$
inside $\mathfrak{M}^1_4(2)$ is given by 
\begin{align*}
\big[\big(\mathsf{R}^1_4\big)_{\mathsf{Fb}}\big] & =      (a+B_1)\cdot (\lambda_1 + B_1) \cdot (\lambda_2 + 2 B_1)\cdot (H_1^3 + H_1^2 B_1 + H_1 B_1^2 + B_1^3) \cdot \prod_{\substack{i=1\\
i \neq 1}}^4\left[(a+H_i) (\lambda_1 + H_i)\right].                                                  
\end{align*}
Similarly, 
\begin{align*}
\big[\big(\mathsf{R}^2_4\big)_{\mathsf{Fb}}\big] & = (a+B_1)\cdot (\lambda_1 + B_1) \cdot (\lambda_2 + 2 B_1)\cdot (H_1^3 + H_1^2 B_1 + H_1 B_1^2 + B_1^3)\cdot \prod_{\substack{i=1\\
i \neq 2}}^4\left[(a+H_i) (\lambda_1 + H_i)\right],                                                 
\end{align*}
and 
\begin{align*}
\big[\big(\mathsf{R}^3_4\big)_{\mathsf{Fb}}\big] & = (a+B_1)\cdot (\lambda_1 + B_1) \cdot (\lambda_2 + 2 B_1)\cdot (H_1^3 + H_1^2 B_1 + H_1 B_1^2 + B_1^3) \cdot \prod_{\substack{i=1\\
i \neq 3}}^4\left[(a+H_i) (\lambda_1 + H_i)\right].   
\end{align*}
All intersection numbers can now be computed. 

\subsection{Nodal conics} 

Here nodal conics are viewed as a pair of unordered lines. We will instead look at the space of ordered lines 
and divide the corresponding intersection number by $2$. The ambient space will be
\[
\mathfrak{M}\ :=\ \big(\mathcal{D}_1 \times \mathcal{D}_1^{\prime} \times \mathcal{D}_1^{\prime \prime} \big)_{\mathsf{Fb}} 
\times Y^1 \times Y_1 \times Y_2 \times Y_3 \times Y_4.
\]
Note that $\mathfrak{M}$ is a subspace of $\mathfrak{M}^1_4(2)$, and the latter is a subspace of $\mathfrak{M}^1_4(3)$.
The classes in $\big(\mathcal{D}_1 \times \mathcal{D}_1^{\prime} \times \mathcal{D}_1^{\prime \prime} \big)_{\mathsf{Fb}}$ 
will be denoted by $\lambda_1,\, \lambda_1^{\prime},\, \lambda_1^{\prime \prime}$ 
and $a$, respectively. The classes in the other factors 
are denoted as usual by $B_1,\, H_1,\, H_2,\, H_3$ and $H_4$, respectively. 
The class $\big[\left(\left(\mathsf{R}\mathsf{A}_1^{\mathsf{F}}\right)_{4}\right)_{\mathsf{Fb}}\big]$ is given by 
\[
\big[\big(\big(\mathsf{R}\mathsf{A}_1^{\mathsf{F}}\big)_{4}\big)_{\mathsf{Fb}}\big]\ =\
 \frac{1}{2}(a+B_1) \cdot (\lambda_1^{\prime} + B_1) \cdot 
(\lambda_1^{\prime \prime} + B_1)  \cdot \prod_{i=1}^4\left[(a+H_i) (\lambda_1 + H_i)\right].
\]
Now note that restricted to this class, we have   
\begin{align}\label{Chern_class_sum_another}
\lambda_3\ & =\ \lambda_1 + \lambda_1^{\prime} + \lambda_1^{\prime \prime}. 
\end{align} 
Using this, we calculate all the relevant intersection numbers.   
 
\subsection{Smooth conics tangent to a line}

We will use the following characterization of tangency. The derivative of the conic along the direction of the line is zero. 
We note that $\big(\big(\mathsf{R}\mathsf{T}_1\big)_4\big)_{\mathsf{Fb}}$ is a subspace of $\mathfrak{M}^1_4(2)$. It is a fiber bundle analogue of the space $\left(\mathsf{R}\mathsf{T}_1 \right)_4$ described in \cref{definition_RT_1}.
An element of this space is a tuple $\left(\Pi,\, \mathsf{L},\, \mathcal{C}_2,\, p,
\, x_1,\, x_2,\,x_3,\,x_4\right)$ such that all the configurations lie on the plane $\Pi$ and $\mathcal{C}_2$
is a smooth conic while $L$ is tangent to $\mathcal{C}_2$ at $p$. We have to find the cycle expression for
the tangency condition. Consider the following space
$$\mathcal{S}\,=\, \{\left([\eta],\,p \right) \,\in\, \widehat{\mathbb{P}}^3 \times Y^1\,\,\big\vert\,\,
\eta(p)\,=\,0 \}.$$
Recall  that $Y^1$ is a copy of $\mathbb{P}^3$.
The natural projection map $\mathcal{S} \,\longrightarrow\,\widehat{\mathbb{P}}^3$ is a fiber bundle. 
Next, assume that $W$ is the relative tangent bundle of $\mathcal{S} \,\longrightarrow
\,\widehat{\mathbb{P}}^3$, that is, $\pi\,:\,W \,\longrightarrow \,\mathcal{S}$ is a rank two vector
bundle such that the fiber over each point $([\eta],\, p)$ is the tangent space of $\mathbb{P}^2_{\eta}$
at the point $p$. The Chern classes of $W$ (cf. \cite[eq. (21)]{BMS_R_N}) are given by 
\[
c_1(W) \,= \,3B_1-a,\ \ \text{ and }\ \ \, c_2(W)\,=\, a^2-2aB_1+3B_1^2.
\]
Consider
$$\mathcal{T}\,:=\, \{\left([\eta],\,\mathsf{L}, \,p\right) \,\in\,
\left(\mathcal{D}_1\right)_{\mathsf{Fb}} \times Y^1\,\,\big\vert\,\,\, \left([\eta],\,p\right)
\,\in\, \mathcal{S},\ p \,\in\, \mathsf{L} \}.$$
Let $\mathbb{L}$ be the line bundle over $\mathcal{T}$ whose fiber over the point
$\left([\eta],\,\mathsf{L},\, p\right)$ is the tangent space of $\mathsf{L}$ at $p$. Then we get the following
short exact sequence of vector bundles over $\mathcal{T}$:
\begin{align*}
0 \,\longrightarrow \,\mathbb{L} \,\longrightarrow\, W \,\longrightarrow \,
\gamma_{\left(\mathcal{D}_1\right)_{\mathsf{Fb}}}^* \otimes \gamma_{Y^1}^* \,\longrightarrow\, 0,
\end{align*}
where the second map is given by $\nabla f_1|_p$ over the point $\left([\eta],\,f_1^{-1}(0),\, p\right)
\,\in \,\mathcal{T}$ and the line is given by a linear homogeneous polynomial $f_1$. It now follows that
the first Chern class of $\mathbb{L}$ is
\begin{align*}
c_1( \mathbb{L})\ &=\ 2B_1-\lambda_1-a.
\end{align*}

For the tangency condition consider the subspace of  $\left(\mathcal{D}_1 \times \mathcal{D}_2\right)_{\mathsf{Fb}} \times Y^1$:
$$\mathcal{U}\,:=\, \{ \left(\Pi,\, \mathsf{L}, \,f_2^{-1}(0),\, p\right)
\,\in\,  \left(\mathcal{D}_1 \times \mathcal{D}_2\right)_{\mathsf{Fb}} \times Y^1\,\,
\big\vert\,\,\,  \left(\Pi,\, \mathsf{L},\, p\right) \,\in\, \mathcal{T},\ f_2(p)\,=\,0\}.$$
Given a tuple $\left(\Pi,\, \mathsf{L}, \,f_2^{-1}(0),\, p\right) $, the conic $f_2^{-1}(0)$ is tangent to
$\mathsf{L}$ at $p$ if and only if 
\begin{align*}
\left(\Pi,\, \mathsf{L}, \,f_2^{-1}(0),\, p\right) \,\in\, \mathcal{U}& & \text{ and } & &
\nabla f_2|_p(u)\,=\,0 \text{ for all }u \,\in\, \mathbb{L}.
\end{align*}
Thus the tangency condition can be expressed as the section of the following bundle 
\begin{align*}
\psi_T\,:\,\mathcal{U} \,\longrightarrow\, \gamma_{\left(\mathcal{D}_2\right)_{\mathsf{Fb}}}^* \otimes \gamma_{Y^1}^{* 2} \otimes \mathbb{L}^*,
\end{align*}
given by $\{\psi_T\left(\Pi,\, \mathsf{L},\, f_2^{-1}(0),\, p\right)\}(f_2\otimes v)\,:=\,\nabla f_2|_p(v)$.
Therefore, the cycle corresponding to the tangency condition in $\left(\mathcal{D}_1 \times
\mathcal{D}_2\right)_{\mathsf{Fb}}$ is given by 
\begin{align*}
& [\mathcal{U}] \cdot e\left(\gamma_{\left(\mathcal{D}_2\right)_{\mathsf{Fb}}}^* \otimes \gamma_{Y^1}^{* 2} \otimes \mathbb{L}^*\right)\\
&\ =\ (a+B_1) \cdot   (\lambda_1 + B_1) \cdot (\lambda_2 + 2 B_1) \cdot \left(a+\lambda_1+\lambda_2\right).
\end{align*} 
Note that $\big(\big(\mathsf{R}\mathsf{T}_1\big)_4\big)_{\mathsf{Fb}}$ is a subspace of $\mathfrak{M}^1_4(2)$ which is again a subspace of $\left(\mathcal{D}_1 \times \mathcal{D}_2\right)_{\mathsf{Fb}} \times Y^1 \times Y_1 \times Y_2 \times Y_3 \times Y_4$.
Thus, we have
\begin{align*}
\big[\big(\big(\mathsf{R}\mathsf{T}_1\big)_4\big)_{\mathsf{Fb}}\big]\ & =\
(a+B_1) \cdot   (\lambda_1 + B_1) \cdot (\lambda_2 + 2 B_1) \cdot   (a+\lambda_1+\lambda_2) \cdot \prod_{i=1}^4\left[(a+H_i) (\lambda_1 + H_i)\right].                                            
\end{align*} 
Therefore, all the relevant intersection numbers of \cref{conics_FB_comp2} can be computed from the above cycle expressions. We use the identities of \cref{Chern_class_sum,Chern_class_sum_another} to calculate them.

\section{Acknowledgements}
We are grateful to Anuvertika Pandey 
for several useful discussions related to the project. 
Indranil Biswas is partially supported by a J. C. Bose Fellowship (JBR/2023/000003). 
Apratim Choudhury 
is funded by the Deutsche Forschungsgemeinschaft (DFG, German Research
Foundation) under Germany´s Excellence Strategy – The Berlin Mathematics
Research Center MATH+ (EXC-2046/1, project ID 390685689, BMS Stipend).
Nilkantha Das is supported by the INSPIRE faculty fellowship (Ref No: IFA21-MA 161) funded by the DST, Govt. 
of India. He would also 
like to thank NISER Bhubaneswar for the hospitality during a visit where a part of the 
project has been carried out.

 \bibliographystyle{siam} 
\bibliography{Degreed}

\end{document}